\documentclass[11pt,a4paper]{amsart}
\usepackage{amssymb}
\usepackage{amsmath}
\usepackage{amsfonts}
\usepackage[english]{babel}
\usepackage[T1]{fontenc}
\usepackage[latin1]{inputenc}
\usepackage{amsthm}
\usepackage[all]{xy}
\usepackage{a4wide}
\xyoption{all}
\usepackage{bbm}
\usepackage{hyperref}
\usepackage{fixltx2e}
\usepackage{bbold}
\usepackage{dsfont}
\usepackage{verbatim}
\usepackage{graphicx}
\usepackage{tikz,tikz-cd}
\usepackage[numbers]{natbib}
\usepackage[shortlabels]{enumitem}
\usepackage{mathrsfs}
\usepackage{adjustbox}
\usetikzlibrary{matrix, arrows}

\newcommand{\N}{\mathcal{N}}

\newcommand{\QA}{\mathfrak{A}}
\newcommand{\K}{\mathcal{K}}

\newcommand{\A}{\mathcal{A}}
\newcommand{\I}{\mathcal{I}}
\newcommand{\J}{\mathcal{J}}
\newcommand{\QN}{\mathcal{QN}}

\newcommand{\FI}{\overline{\mathcal{F}(X)}^{\Vert\cdot\Vert_\mathcal I}}

\newcommand{\QAQN}{\mathfrak{A}_X^{\mathcal{QN}_p}}
\newcommand{\QASK}{\mathfrak{A}_X^{\mathcal{SK}_p}}
\newcommand{\SK}{\mathcal{SK}}
\newcommand{\F}{\mathcal{F}}
\newcommand{\V}{\Vert\cdot\Vert}
\newcommand{\VI}{\Vert\cdot\Vert_\mathcal I}

\theoremstyle{definition}
\newtheorem{thm}{Theorem}[section]
\newtheorem{prop}[thm]{Proposition}
\newtheorem{lem}[thm]{Lemma}

\newtheorem{qu}[thm]{Question}

\newtheorem{remark}[thm]{Remark}
\newtheorem{remarks}[thm]{Remarks}
\newtheorem{ex}[thm]{Example}
\newtheorem{fact}[thm]{Fact}

\theoremstyle{plain}
\numberwithin{equation}{section}

\title[Quotient algebras of Banach operator ideals]{Quotient algebras of Banach operator ideals related to non-classical approximation properties}
\author[Henrik Wirzenius]{Henrik Wirzenius}
\address{Henrik Wirzenius: Department of Mathematics and Statistics, Box 68, Pietari Kalmin katu 5,
FI-00014 University of Helsinki, Finland}
\email{henrik.wirzenius@helsinki.fi}
\thanks{This work was supported by the Magnus Ehrnrooth Foundation and the Swedish Cultural Foundation in Finland}
\begin{document}
\begin{abstract}
We investigate the quotient algebra $\mathfrak{A}_X^{\mathcal I}:=\mathcal I(X)/\overline{\mathcal F(X)}^{||\cdot||_{\mathcal I}}$ for Banach operator ideals $\mathcal I$ contained in the ideal of the compact operators, where $X$ is a Banach space that fails the $\mathcal I$-approximation property. The main results concern the nilpotent quotient algebras $\mathfrak A_X^{\mathcal{QN}_p}$ and $\mathfrak A_X^{\mathcal{SK}_p}$ for the quasi $p$-nuclear operators $\mathcal{QN}_p$ and the Sinha-Karn $p$-compact operators $\mathcal{SK}_p$. The results include the following: (i) if $X$ has cotype 2, then $\mathfrak A_X^{\mathcal{QN}_p}=\{0\}$ for every $p\ge 1$; (ii) if $X^*$ has cotype 2, then $\mathfrak A_X^{\mathcal{SK}_p}=\{0\}$ for every $p\ge 1$; (iii) the exact upper bound of the index of nilpotency of $\mathfrak A_X^{\mathcal{QN}_p}$ and $\mathfrak A_X^{\mathcal{SK}_p}$ for $p\neq 2$ is $\max\{2,\left \lceil p/2 \right \rceil\}$, where $\left \lceil p/2 \right \rceil$ denotes the smallest $n\in\mathbb N$ such that $n\ge p/2$; (iv) for every $p>2$ there is a closed subspace $X\subset c_0$ such that both $\mathfrak A_X^{\mathcal{QN}_p}$ and $\mathfrak A_X^{\mathcal{SK}_p}$ contain a countably infinite decreasing chain of closed ideals. In addition, our methods yield a closed subspace $X\subset c_0$ such that the compact-by-approximable algebra $\mathfrak A_X=\mathcal K(X)/\mathcal A(X)$ contains two incomparable countably infinite chains of nilpotent closed ideals. 
\end{abstract}

\maketitle

\section{Introduction}\label{section1}
Let $X$ be a Banach space and let $\I=(\I,\VI)$ be a Banach operator ideal. The closure $\FI$ is a $\VI$-closed ideal of $\I(X)$, where $\F(X)$ consists of the bounded finite-rank operators $X\to X$, and consequently, the quotient algebra $\QA_X^{\I}:=\I(X)/\overline{\F(X)}^{\VI}$ is a Banach algebra when equipped with the quotient norm 
\[||T+\overline{\F(X)}^{\VI}||=\inf\big\{||T-S||_\I\mid S\in\overline{\F(X)}^{\VI}\big\},\quad T\in\I(X).\]
In the present paper we initiate a study of the quotient algebra $\QA_X^{\I}$ for Banach operator ideals $\I=(\I,\VI)$ contained in the ideal of the compact operators $\K=(\K,\V)$. For such Banach operator ideals $\I$, the non-unital quotient algebra $\QA_X^\I$ is radical, and it can only be non-trivial within the class of Banach spaces $X$ failing the $\I$-approximation property. Here the Banach space $X$ is said to have the $\I$-approximation property ($\I$-AP in short) if $\F(Y,X)$ is $\VI$-dense in $\I(Y,X)$ for every Banach space $Y$. The $\I$-AP has recently been studied, under varying terminology, for different Banach operator ideals $\I=(\I,\VI)$, see e.g. \cite{SK02, DPS10_2, CK10, GLT12, LT13, LT17, Kim19, Kim20} and their references.

It is well known that the $\K$-AP coincides with the classical approximation property (AP in short) of Grothendieck. Moreover, the $\I$-AP is typically a weaker property than the AP and it is therefore in general difficult to construct or to identify examples of Banach spaces failing the $\I$-AP. For $\I=\K$ one obtains the quotient algebra $\QA_X:=\QA_X^\K$ of compact-by-approximable operators, that is, $\QA_X=\K(X)/\A(X)$ where the uniform closure $\A(X)=\overline{\F(X)}$ denotes the class of approximable operators $X\to X$. The compact-by-approximable algebra $\QA_X$  has been systematically studied in \cite{TW21, TW22} and one of the objectives of this paper is to exhibit different phenomena for the class of radical quotient algebras $\QA_X^\I$ which are not known to manifest themselves for the compact-by-approximable algebra $\QA_X$ for any Banach space $X$. 

The primary focus is on the quotient algebra $\QA_X^{\QN_p}=\QN_p(X)/\overline{\F(X)}^{\V_{\QN_p}}$ for the Banach operator ideal $\QN_p=(\QN_p,\V_{\QN_p})$ of quasi $p$-nuclear operators \cite{PP69}, and the quotient algebra $\QA_X^{\SK_p}=\SK_p(X)/\overline{\F(X)}^{\V_{\SK_p}}$ for the Banach operator ideal $\SK_p=(\SK_p,\V_{\SK_p})$ of Sinha-Karn $p$-compact operators \cite{SK02}, where $p\in[1,\infty)$ is a real number. In recent years, the class $\SK_p$ for $p\in[1,\infty)$ has arguably been the most studied class of non-classical Banach operator ideals with respect to the $\I$-AP, see e.g. \cite{DPS10_2, GLT12, Oja12, Kim15}. It is known that in an isometric sense both $\QN_p=\SK_p^{dual}$ and $\SK_p=\QN_p^{dual}$, where $\I^{dual}$ denotes the dual ideal of a Banach operator ideal $\I$, and therefore a simultaneous study of the quotient algebras $\QAQN$ and $\QASK$ is relevant. 

It follows from a classical result of Persson and Pietsch \cite[Satz 48]{PP69} that both quotient algebras $\QAQN$ and $\QASK$ are nilpotent for any given Banach space $X$. Moreover, it turns out that $\QAQN$ and $\QASK$ are trivial for large classes of Banach spaces $X$, which include spaces without the AP. In fact, in Section \ref{section3} we show that if $X$ has cotype 2 (respectively, if $X^*$ has cotype 2), then $\QAQN=\{0\}$ (respectively, $\QASK=\{0\}$) for every $1\le p<\infty$. We also provide for every $1\le p<2$ examples of Banach spaces $X$ that fail the $\QN_p$-AP, but for which $\QAQN=\{0\}$. Analogously, for every $2<p<\infty$ there are  Banach spaces $X$ that fail the $\SK_p$-AP for which $\QASK=\{0\}$. Recall here that it is not known whether there exists a Banach space $X$ that fails the AP for which $\QA_X=\{0\}$.

In Section \ref{section4} we discuss the nilpotency of the quotient algebras $\QAQN$ and $\QASK$. The aforementioned result of Persson and Pietsch yields the upper bound $\max\{2,\left \lceil p/2 \right \rceil\}$ for the index (of nilpotency) of $\QAQN$ and $\QASK$ for any Banach space $X$. Here $\left \lceil p/2\right \rceil=\min\{n\in\mathbb N\mid n\ge p/2\}$ and the index of a nilpotent algebra $A$ is the minimum of all natural numbers $n\in\mathbb N$ for which the product $a_1\cdots a_n=0$ whenever $a_1,\ldots, a_n\in A$. Expanding on a sophisticated factorisation method due to  Reinov \cite{Reinov82}, we establish that the upper bound $\max\{2,\left \lceil p/2 \right \rceil\}$ is exact in the sense that for every $p\neq 2$ there is a closed subspace $X\subset c_0$ for which the index of both quotient algebras $\QAQN$ and $\QASK$ equals $\max\{2,\left \lceil p/2 \right \rceil\}$.

In Section \ref{section5} we provide examples of Banach spaces $X$ for which the quotient algebras $\QAQN$ and $\QASK$ carry non-trivial closed ideals of a specific natural form. By a closed ideal of a given Banach algebra $A=(A,\V_A)$ we always mean a $\V_A$-closed two-sided ideal of $A$. Our methods also provide new examples of closed ideals of the compact-by-approximable algebra $\QA_X$ that improve upon a result obtained in \cite{TW22}. Namely, our main theorem (Theorem \ref{2083}) exhibits a closed subspace $X\subset c_0$ for which $\QA_X$ carries two incomparable countably infinite  chains of nilpotent closed ideals. At the same time, we obtain a countably infinite chain of closed ideals of the quotient algebras $\QAQN$ and $\QASK$, respectively, for any given $2<p<\infty$.

The non-trivial closed ideals found in Section \ref{section5} arise from very natural closed ideals of $\QN_p(X)$ and $\SK_p(X)$, respectively, which are all in fact algebraic ideals of the bounded operators $\mathcal L(X)$.  
In Section \ref{section6} we exhibit a particular construction of a Banach space $Z$ such that $\I(Z)$ carries an uncountable family of closed ideals that are not ideals of $\mathcal L(Z)$. Here $\I=(\I,\VI)$ can be any Banach operator ideal contained in the compact operators $\K=(\K,\V)$ for which the $\I$-AP fails a certain duality property. We also discuss and draw attention to some duality problems related to the quotient algebras $\QAQN$ and $\QASK$.

\section{Preliminary results and examples}\label{section2}
In this section we recall relevant definitions and concepts related to Banach operator ideals $\I=(\I,\VI)$ and the corresponding $\I$-approximation property. Moreover, we show that the quotient algebra $\QA_X^\I$ is a radical Banach algebra whenever the Banach operator ideal $\I$ is contained in the ideal of the compact operators $\K=(\K,\V)$. We also briefly indicate the general connection between the $\I$-approximation property and the quotient algebra $\QA_X^\I$, and describe a basic construction of a Banach space $Z$ for which dim$\,\QA_Z^\I=\infty$ for any non-approximative Banach operator ideal $\I=(\I,\VI)$. We refer to  \cite{Pietsch80}, \cite{DF93} and the survey \cite{DJP} for comprehensive sources on Banach operator ideals. For unexplained notions in Banach space theory we refer to \cite{DJT95} and \cite{AK06}.
Throughout the paper we consider Banach spaces over the same scalar field $\mathbb K=\mathbb R$ or $\mathbb K=\mathbb C$. 
\smallskip

Let $\mathcal L(X,Y)$ denote the space of all bounded linear operators $X\to Y$. We use the term \emph{Banach operator ideal}  for a complete normed operator ideal in the sense of Pietsch \cite{Pietsch80}. Recall that a Banach operator ideal $\I=(\I,\VI)$ assigns for each pair of Banach spaces $(X,Y)$ a linear subspace $\I(X,Y)$ of $\mathcal L(X,Y)$, and each component $\I(X,Y)$ is equipped with a complete norm $\VI$ so that the following properties hold:
\begin{enumerate}
\item[(BOI1)] $x^*\otimes y\in\mathcal I(X,Y)$ and $||x^*\otimes y||_\I=||x^*||\,||y||$ for all $x^*\in X^*$ and $y\in Y$.
\item[(BOI2)] for all Banach spaces $Z$ and $W$,  and all operators $A\in\mathcal L(Z,X)$, $B\in\mathcal L(Y,W)$ and $S\in\I(X,Y)$ one has $BSA\in \I(Z,W)$. In addition, $||BSA||_\I\leq ||B||\, ||A||\, ||S||_\I$.
\item[(BOI3)] $||S||\leq ||S||_\I$ for all $S\in\I(X,Y)$.
\end{enumerate}
Above $x^*\otimes y$ denotes the bounded rank-one operator $x\mapsto x^*(x)y$ and $\V$ is the uniform operator norm. The first part of (BOI2) will be referred to as the \emph{operator ideal property}. 

For two Banach operator ideals $\I=(\I,\VI)$ and $\J=(\J,\V_{\J})$ we write
\[\I\subset\mathcal J\] 
if $\I(X,Y)\subset \mathcal J(X,Y)$ for all Banach spaces $X$ and $Y$, and $||\cdot||_{\mathcal J}\leq ||\cdot||_\I$. If in turn $\I(X,Y)=\J(X,Y)$ for all Banach spaces $X$ and $Y$, and $||\cdot||_\J=||\cdot||_\I$, we write $\I=\J$.
\smallskip

Let $\I=(\I,\VI)$ be an arbitrary Banach operator ideal and let $X$ be a Banach space.  We say that the quotient algebra $\QA_X^\I$ is radical if the spectrum
\begin{equation}\label{radical}
\sigma(T+\overline{\F(X)}^{\VI}):=\{\lambda\in\mathbb K\mid \lambda \mathbb 1-(T+\overline{\F(X)}^{\VI})\text{ is not invertible in }(\QA_X^\I)^{\#}\}=\{0\}
\end{equation}
for all $T\in\I(X)$.
Here $(\QA_X^\I)^{\#}=\mathbb K\oplus \QA_X^\I$ is the formal unitisation of the Banach algebra $\QA_X^\I$ and $\mathbb 1$ denotes the unit element of $(\QA_X^\I)^{\#}$.  Recall that the condition \eqref{radical} is equivalent to the usual definition of a radical Banach algebra whenever the underlying scalar field $\mathbb K$ is the complex numbers $\mathbb C$, see e.g.  \cite[Proposition 1.5.32(iv)]{Dales00}. Moreover, if $\mathbb K=\mathbb R$, then \eqref{radical} is consistent with the interpretation of a radical real compact-by-approximable algebra $\QA_X$ displayed on page 4 of \cite{TW22}.

Our first result verifies that the quotient algebra $\QA_X^\I$ is radical in the sense of \eqref{radical} for any Banach operator ideal $\I=(\I,\VI)$ contained in the compact operators $\K=(\K,\V)$. Recall here that the compact-by-approximable algebra $\QA_X$ is a radical Banach algebra, see \cite[Theorem 2.5.8(iv)]{Dales00} for the case $\mathbb K=\mathbb C$ and \cite[Proposition 2.8]{TW22} for the case $\mathbb K=\mathbb R$.

\begin{prop}\label{prop21}
Suppose that $\I=(\I,\VI)$ is a Banach operator ideal such that $\I\subset\K$ and let $X$ be an arbitrary Banach space. Then the quotient algebra $\QA_X^\I$ is a radical Banach algebra in the sense of \eqref{radical}.
\end{prop}
\begin{proof}
Suppose that $\lambda\in\mathbb K\setminus\{0\}$ and let $T\in\I(X)$ be arbitrary. Since $\I\subset\K$, we have $\lambda^{-1}T\in\K(X)$. According to the Riesz-Schauder theory (see the implication (c)$\Rightarrow$(b) of \cite[Theorem 1.4.5]{CPY74}), there is an invertible bounded operator $S\in\mathcal L(X)$ and a bounded finite-rank operator $F\in\mathcal F(X)$ such that
\begin{equation}\label{166}
\lambda^{-1}T-I_X=S+F.
\end{equation} 
Here $I_X$ denotes the identity operator on $X$. We claim that
\begin{equation}\label{rad}
\lambda^{-1}\mathbb 1-(\lambda^{-2}TS^{-1}+\FI)\in(\QA_X^\I)^{\#}
\end{equation} 
is the inverse of the element $\lambda \mathbb 1-(T+\FI)\in(\QA_X^\I)^{\#}$. Note here that $TS^{-1}\in\I(X)$ by the operator ideal property and thus the element in \eqref{rad} is a well-defined element of the unitisation $(\QA_X^\I)^{\#}$. Towards the claim, we first observe that the identity \eqref{166} implies that
\begin{equation}\label{k0}
-T-TS^{-1}+\lambda^{-1}T^2S^{-1}=-T(S+I_X-\lambda^{-1}T)S^{-1}=TFS^{-1}\in \F(X)
\end{equation}
and
\begin{equation}\label{k12}
-TS^{-1}-T+\lambda^{-1}TS^{-1}T=-TS^{-1}(I_X+S-\lambda^{-1}T)=TS^{-1}F\in\F(X),
\end{equation}
since $F\in \F(X)$. It then follows from \eqref{k0} that
\begin{align*}
\Big( \lambda \mathbb{1}&-(T+\FI)\Big)\cdot\Big(\lambda^{-1}\mathbb{1}-(\lambda^{-2}TS^{-1}+\FI)\Big)\\
&=\mathbb{1}+(\lambda^{-1} TFS^{-1}+\FI)=\mathbb{1} \in(\QA_X^\I)^{\#},
\end{align*}
and similarly from \eqref{k12} that 
\begin{align*}
\Big(\lambda^{-1}&\mathbb{1}-(\lambda^{-2}TS^{-1}+\FI)\Big)\cdot\Big(\lambda \mathbb{1}-(T+\FI)\Big)\\
&=\mathbb 1+(\lambda^{-1}TS^{-1}F+\FI)=\mathbb 1\in(\QA_X^\I)^{\#}.
\end{align*}
Thus the element $\lambda \mathbb 1-(T+\FI)$ is invertible in $(\QA_X^\I)^{\#}$, and so $\lambda\notin\sigma(T+\FI)$. Since it is clear that $0\in\sigma(T+\FI)$, we thus have $\sigma(T+\FI)=\{0\}$.
\end{proof}

\begin{remark}
Recall that a complex radical Banach algebra cannot have a unit. It follows from Proposition \ref{prop21} that also in the real case $\mathbb K=\mathbb R$ the quotient algebra $\QA_X^\I$ is non-unital whenever $\I\subset\K$. In fact, assume towards a contradiction that $U+\FI\in\QA_X^\I$ is the unit element, where $U\in\I(X)\setminus\FI$. By Proposition \ref{prop21} we have $1\notin\sigma(T+\FI)$ and thus the element $\mathbb 1-(U+\FI)\in (\QA_X^\I)^{\#}$ is invertible. Let $\lambda \mathbb 1+(T+\FI)$ be the inverse for some $\lambda\in\mathbb K$ and $T\in\I(X)$. It follows that
\begin{align*}
\mathbb 1&=\Big(\mathbb 1-(U+\FI)\Big)\cdot\Big(\lambda\mathbb 1+(T+\FI)\Big)\\
&=\lambda\mathbb  1-(\lambda U+T-UT+\FI)
=\lambda \mathbb 1-(\lambda U+\FI),
\end{align*}
where the last equality holds since $T-UT\in\FI$. This means that $\lambda=1$ and $U\in\FI$, which is a contradiction.
\end{remark}
We proceed with a brief discussion of the general connection between the $\I$-approximation property and the quotient algebra $\QA_X^\I$ for an arbitrary Banach operator ideal $\I=(\I,\VI)$. In Section \ref{section3} we explore this connection further for the Banach operator ideals $\QN_p$ of quasi $p$-nuclear operators and $\SK_p$ of Sinha-Karn $p$-compact operators.

Let $\I=(\I,\VI)$ be a Banach operator ideal. Following Delgado et al.  \cite{DPS10_2} and Oja \cite{Oja12}, we say that the Banach space $X$ has the \emph{$\mathcal I$-approximation property} ($\mathcal I$-AP) if 
\begin{equation*}
\I(Y,X)=\overline{\F(Y,X)}^{\V_\mathcal I}
\end{equation*} 
for every Banach space $Y$. The notion of the $\I$-AP is a natural version of the classical approximation property related to an arbitrary Banach operator ideal $\I=(\I,\VI)$. In fact, recall that the Banach space $X$ has the \emph{approximation property} (AP) if for all compact subsets $K\subset X$ and all $\varepsilon>0$ there is a bounded finite-rank operator $F\in\F(X)$ such that
\[\sup_{x\in K}||Fx-x||<\varepsilon.\]
Due to a classical result of Grothendieck (see e.g. \cite[Theorem 1.e.4]{LT77}), the Banach space $X$ has the AP if and only if $\K(Y,X)=\A(Y,X)$ for every Banach space $Y$, where $\A=\overline{\F}$ denotes the Banach operator ideal $\A=(\A,\V)$ of \emph{approximable} operators. Consequently, $X$ has the AP if and only if $X$ has the $\K$-AP.
\smallskip

Clearly $\QA_X^\I=\{0\}$ whenever $X$ has the $\I$-AP, and thus a good understanding of the $\I$-AP is essential for studies of the quotient algebra $\QA_X^\I$. In general, it may happen that $\QA_X^\I=\{0\}$ even though $X$ fails the $\I$-AP as we will see in Example \ref{trivialquotient2}. Nevertheless, if the Banach space $X$ fails the $\I$-AP, there is by definition a Banach space $Y$ such that $\I(Y,X)\neq \overline{\F(Y,X)}^{\VI}$. It then follows by the operator ideal property that $\QA_{X\oplus Y}^\I\neq \{0\}$ as we will demonstrate below.

For this, pick an operator $T\in\I(Y,X)\setminus\overline{\F(Y,X)}^{\VI}$. Let $J_X:X\to X\oplus Y$ and $J_Y:Y\to X\oplus Y$ denote the natural isometric embeddings, and let $P_X:X\oplus Y\to X$ and $P_Y:X\oplus Y\to Y$ denote the natural projections. Consider the operator $U:=J_XTP_Y\in\mathcal L(X\oplus Y)$. By the operator ideal property (for $\I$) we have $U\in\I(X\oplus Y)$. On the other hand, by the operator ideal property for the approximative kernel $\overline{\F}^{\VI}$ (see below for the definition of $\overline{\F}^{\VI}$) we have $U\notin \overline{\F(X\oplus Y)}^{\VI}$. In fact, suppose on the contrary that $U\in\overline{\F(X\oplus Y)}^{\VI}$. Then $T=P_XUJ_Y\in\overline{\F(Y,X)}^{\VI}$ by the operator ideal property, which is a contradiction.

Above we used the fact that the class $\overline{\F}^{\VI}=(\overline{\F}^{\VI},\VI)$ defined by the components $\overline{\F}^{\VI}(X,Y)=\overline{\F(X,Y)}^{\VI}$ and the ideal norm $\VI$ is a Banach operator ideal, which is called the \emph{approximative kernel} of $\I=(\I,\VI)$, see \cite[D.1.13]{Pietsch87}.
If $\I=\overline{\F}^{\VI}$, then $\I$ is called an \emph{approximative} Banach operator ideal. Observe that there exists Banach spaces that fail the $\I$-AP only for non-approximative Banach operator ideals $\I=(\I,\VI)$.
\smallskip

The simple construction above of the direct sum $X\oplus Y$ with a non-trivial quotient algebra $\QA_{X\oplus Y}^{\I}$ can be generalised to construct a Banach space $Z$ for which $\QA_Z^\I$ is infinite-dimensional. We refer to \cite[Section 2]{TW21} for related results on the compact-by-approximable algebra $\QA_X=\K(X)/\A(X)$. 
\begin{ex}\label{infinitedimensional}
Let $\I=(\I,\VI)$ be a non-approximative Banach operator ideal and suppose that the Banach space $X$ fails the $\I$-AP. Let $Y$ be a Banach space such that 
\begin{equation}\label{notequal}
\I(Y,X)\neq \overline{\F(Y,X)}^{\VI}.
\end{equation} 
Denote $X_0=X$ and $X_n=Y$ for every $n\in\mathbb N$, and consider the $\ell^p$-direct sum $Z:=\left(\bigoplus_{n=0}^\infty X_n\right)_{\ell^p}$, where $1\le p\le\infty$ and the case $p=\infty$ denotes a direct $c_0$-sum with the supremum norm. Then the quotient algebra $\QA_Z^\I$ is infinite-dimensional.
\end{ex}
\begin{proof}
For each $n\in\mathbb N_0:=\mathbb N\cup\{0\}$ let $P_n:Z\to X_n$ denote the natural projection and let $J_n:X_n\to Z$ denote the corresponding natural isometric embedding. In view of \eqref{notequal} there is an operator $T\in\I(Y,X)\setminus\overline{\F(Y,X)}^{\VI}$ so that $c:=\inf||T-A||_{\I}>0$, where the infimum is taken over all $A\in\overline{\F(Y,X)}^{\VI}$. Let $T_n:=J_0 TP_n\in\I(Z)$ for all $n\in\mathbb N$. Note that
\begin{equation}\label{bounded}
||T_n+\overline{\F(Z)}^{\VI}||\le||T_n||_\I=||J_0TP_n||_\I\le||T||_\I
\end{equation}
for all $n\in\mathbb N$ by (BOI2).
Moreover, for all $m,n\in\mathbb N$, $n\neq m$ and all $S\in\overline{\F(Z)}^{\VI}$ we have
\begin{align*}
||T_n-T_m-S||_{\I}&=||J_0TP_n-J_0TP_m-S||_{\I}\\
&\geq||P_0(J_0TP_n-J_0TP_m-S)J_n||_{\I}\\
&=||T-P_0SJ_n||_{\I}\ge c>0
\end{align*}
since $P_mJ_n=0$ and $P_0SJ_n\in\overline{\F(Y,X)}^{\VI}$ by the operator ideal property. It follows that
\begin{equation}\label{subsequences}
||T_n-T_m+\overline{\F(Z)}^{\VI}||\ge c>0
\end{equation} 
for all $n\neq m$. By \eqref{bounded} and \eqref{subsequences} the sequence $(T_n+\overline{\F(Z)}^{\V_{\I}})\subset\mathfrak A_Z^{\I}$ is a bounded sequence with no  convergent subsequences. Consequently, dim$\,\QA_Z^\I=\infty$.
\end{proof}
\section{Quasi \texorpdfstring{$p$}{p}-nuclear operators and Sinha-Karn \texorpdfstring{$p$}{p}-compact operators}\label{section3}
In this section we initiate our studies of the quotient algebras $\QAQN$ and $\QASK$
for the quasi $p$-nuclear operators $\QN_p=(\QN_p,\V_{\QN_p})$ and the Sinha-Karn $p$-compact operators $\SK_p=(\SK_p,\V_{\SK_p})$. It turns out that these quotient algebras are trivial for large classes of Banach spaces. In fact, in the main result of this section (Theorem \ref{trivialquotient}) we establish that if the Banach space $X$ has cotype 2, then the quotient algebra $\QAQN=\{0\}$ for every $1\le p<\infty$. Similarly, if the dual space $X^*$ has cotype 2, then $\QASK=\{0\}$ for every $1\le p<\infty$. These results involve observations on the $\QN_p$-AP and the $\SK_p$-AP which are of independent interest. Moreover, our results yield Banach spaces $X$ that fail the $\QN_p$-AP, but for which $\QAQN=\{0\}$ in the case of $1\le p<2$. Similar examples is obtained for the Sinha-Karn $p$-compact operators $\SK_p$ in the case of $2<p<\infty$. 
\smallskip

We proceed with the definitions, and a discussion of relevant properties of the Banach operator ideals $\QN_p$ and $\SK_p$. Throughout this section the closed unit ball of the Banach space $X$ is denoted by $B_X$ and $p'$ denotes the dual exponent of the real number $p\in[1,\infty)$.

Suppose that $1\leq p<\infty$ and let $X$ and $Y$ be Banach spaces. Following Persson and Pietsch \cite[Section 4]{PP69}, the operator $T\in\mathcal L(X,Y)$ is called \emph{quasi $p$-nuclear}, denoted $T\in\QN_p(X,Y)$, if there is a strongly $p$-summable sequence $(x_k^*)\in\ell^p_s(X^*)$ such that
\begin{equation}\label{rt}
||Tx||\leq \big(\sum_{k=1}^\infty |x_k^*(x)|^p\big)^{1/p},\quad x\in X.
\end{equation}
The class $\QN_p=(\QN_p,\V_{\QN_p})$ is a Banach operator ideal where the ideal norm $\V_{\QN_p}$ is defined by
\[||T||_{\QN_p}=\inf\{||(x_k^*)||_p\mid (\ref{rt})\text{ holds for }(x_k^*)\in\ell^p_s(X^*)\},\qquad T\in\QN_p(X,Y).\]
\begin{remark}\label{QNpinjhull}
Let $1\le p<\infty$. It is known that $\QN_p=\N_p^{inj}$, where $\N_p^{inj}$ denotes the injective hull of the $p$-nuclear operators $\N_p=(\N_p,\V_{\N_p})$, see \cite[Satz 38 and Satz 39]{PP69} or the comment before \cite[Theorem 6]{Pietsch14}.
\end{remark}
We say that the operator $T\in\mathcal L(X,Y)$ is a \emph{Sinha-Karn $p$-compact} operator, denoted $T\in\SK_p(X,Y)$, if there is a strongly $p$-summable sequence $(y_k)\in\ell^p_s(Y)$ such that
\begin{equation}\label{t}
T(B_X)\subset \Big\{\sum_{k=1}^\infty \lambda_k y_k\mid (\lambda_k)\in B_{\ell^{p'}}\Big\}\quad\big(\text{if }p=1\text{ then }(\lambda_k)\in B_{c_0}\big).
\end{equation}
The class $\SK_p=(\SK_p,\V_{\SK_p})$ of the Sinha-Karn $p$-compact operators is a Banach operator ideal, see \cite[Theorem 4.2]{SK02} and \cite[Proposition 3.15]{DPS10}, where the ideal norm $\V_{\SK_p}$ is defined by
\[||T||_{\SK_p}=\inf \{||(y_k)||_p\mid (\ref{t})\text{ holds for }(y_k)\in\ell^p_s(Y)\},\qquad T\in\SK_p(X,Y).\] 
We point out that the Sinha-Karn $p$-compact operators were introduced in \cite{SK02} as $p$-compact operators and with the notation $(K_p,\kappa_p)$ for the corresponding Banach operator ideal. However, we prefer to follow the slightly modified notation and terminology introduced by Tylli and the author in \cite{TW22}.

We will require the following facts: 
\begin{align}\label{monot}
&\QN_p\subset\QN_q\subset\K,\\
\label{monotonicity} &\SK_p\subset\SK_q\subset\mathcal K,
\end{align} 
for all $1\leq p\leq q<\infty$, see \cite[Satz 24 and Satz 25]{PP69} for \eqref{monot} and \cite[Proposition 4.3]{SK02} or \cite[p. 949]{Oja12} for \eqref{monotonicity}. Moreover, the following dualities hold for all Banach spaces $X$ and $Y$, and all operators $T\in\mathcal L(X,Y)$, see \cite[Theorem 2.8]{GLT12}:
\begin{align}
T\in\QN_p(X,Y)&\Leftrightarrow T^*\in\SK_p(Y^*,X^*).\text{ In this case }||T||_{\QN_p}=||T^*||_{\SK_p}. \label{911}\\
T\in\SK_p(X,Y)&\Leftrightarrow T^*\in\QN_p(Y^*,X^*).\text{ In this case }||T||_{\SK_p}=||T^*||_{\QN_p}.\label{9112}
\end{align}
In other words, $\QN_p=\SK_p^{dual}$ and $\SK_p=\QN_p^{dual}$, where $\I^{dual}=(\I^{dual},\V_{\I^{dual}})$ denotes the dual ideal of the Banach operator ideal $\I=(\I,\VI)$. 
Recall that the components of the dual ideal $\I^{dual}$ are defined by 
\[\I^{dual}(X,Y)=\{T\in\mathcal L(X,Y)\mid T^*\in\I(Y^*,X^*)\},\]
and the associated ideal norm is defined by $||T||_{\I^{dual}}=||T^*||_{\I}$ for $T\in\I^{dual}(X,Y)$.

Due to the dualities \eqref{911} and \eqref{9112}, a simultaneous study of the quotient algebras $\QAQN$ and $\QASK$ is relevant. In fact, it is easy to check that if $X$ is a reflexive Banach space, then the mapping
\[\Phi_X:\QAQN\to \QA_{X^*}^{\SK_p}, \quad \Phi_X\big(T+\overline{\F(X)}^{\V_{\QN_p}}\big)=T^*+\overline{\F(X^*)}^{\V_{\SK_p}}\]
defines an isometric linear isomorphism which reverses products for any $1\le p<\infty$. However, this mapping is not a linear isomorphism in general for non-reflexive spaces in the case $p\neq 2$, and in Section \ref{section6} we will address some of the subtleties of the relationship between $\QAQN$ and $\QA_{X^*}^{\SK_p}$ as well as between $\QASK$ and $\QA_{X^*}^{\QN_p}$.
\smallskip

We proceed with results and examples concerning the $\QN_p$-AP and the $\SK_p$-AP. The study of the $\SK_p$-AP was initiated in \cite{DPS10_2} and it has thereafter been studied in several papers, see e.g. \cite{GLT12, Oja12, LT13, Kim15}. Moreover, the $p$-AP, which is a related weaker approximation property, has been studied in e.g. \cite{SK02, DOPS09, CK10}. The $\QN_p$-AP has received less attention, although it is known that the $\QN_p$-AP coincides with the \emph{approximation property of type p} whose study was initiated by Sinha and Karn \cite{SK08}, see the discussion on \cite[p. 499]{DPS10_2}. 

\begin{remarks}\label{9113}
(i) If the Banach space $X$ has the AP, then $X$ has both the $\QN_p$-AP and the $\SK_p$-AP for every $1\leq p<\infty$. Here the case of the $\QN_p$-AP follows from \cite[Satz 26]{PP69} (see \cite[19.7.7]{Jarchow81} for a detailed proof of \cite[Satz 26]{PP69}) or \cite[Proposition 4.8]{SK08}, and the case of the $\SK_p$-AP was established in \cite[Proposition 3.10]{GLT12}. 
\smallskip

(ii) The Banach operator ideals $\QN_2$ and $\SK_2$ are approximative, that is, $\QN_2=\overline{\F}^{\V_{\QN_2}}$ and $\SK_2=\overline{\F}^{\V_{\SK_2}}$, and thus every Banach space has both the $\QN_2$-AP and the $\SK_2$-AP. These facts can be verified in different ways. A short proof for $\SK_2$ is indicated on \cite[p. 952]{Oja12} and the case of $\QN_2$ can be shown in a similar way. In fact, it follows from \cite[18.1.8 and 18.2.1]{Pietsch80} that $\mathcal N_2^{inj}=\mathcal N_2$ and thus $\QN_2=\mathcal N_2$ (see Remark \ref{QNpinjhull}). It follows that $\QN_2$ is approximative since $\N_p$ is approximative for any $1\le p<\infty$ according to \cite[Satz 5]{PP69}.
\end{remarks}
Let $1\le p<\infty$ and $p\neq 2$. Since the AP implies both the $\QN_p$-AP and the $\SK_p$-AP, it is in general a difficult task to construct or recognise Banach spaces that fail either the $\QN_p$-AP or the $\SK_p$-AP. 
Due to results of Reinov \cite{Reinov82} that rely on Enflo's \cite{E} and Davie's \cite{Da73} constructions of Banach spaces failing the AP, there is a Banach space that fails the $\QN_p$-AP, see also \cite[p. 164]{DOPS09}. By duality, there is a Banach space that fails the $\SK_p$-AP, see \cite[Theorem 2.4]{DPS10_2}. In Examples \ref{Ex:33} and \ref{Ex:34} we deduce that such spaces exist within the class of the closed subspaces of $\ell^p$ from known results on these approximation properties and the classical AP. In addition, by applying Terzio\u{g}lu's factorisation theorem \cite{Terzioglu71} and Example \ref{infinitedimensional}, we find closed subspaces $Z_1,Z_2\subset c_0$ such that the quotient algebras $\QA_{Z_1}^{\QN_p}$ and $\QA_{Z_2}^{\SK_p}$ are infinite-dimensional.

For the proofs of Example \ref{Ex:33}.(i) and Proposition \ref{SKp}, we recall that the operator $T\in\mathcal L(X,Y)$ is called \emph{absolutely $p$-summing}, denoted $T\in\varPi_p(X,Y)$, if there is a constant $C\ge 0$ such that 
\begin{equation}\label{121}
\big(\sum_{k=1}^n ||Tx_k||^p\big)^{1/p}\le C\cdot\sup_{x^*\in B_{X^*}}\big(\sum_{k=1}^n |x^*(x_k)|^p\big)^{1/p}
\end{equation}
for every finite subset $\{x_1,\ldots,x_n\}\subset X$. 
The class $\varPi_p=(\varPi_p,\V_{\varPi_p})$ is a Banach operator ideal for all $1\le p<\infty$, where the absolutely $p$-summing norm $||T||_{\varPi_p}$ of any $T\in\varPi_p(X,Y)$ is the infimum of all constants $C\ge 0$ for which \eqref{121} holds, see e.g. \cite[Remark 11.1]{DF93}. A straightforward computation yields the following Banach operator ideal inclusion:
\begin{equation}\label{0311}
\QN_p\subset\varPi_p
\end{equation}
for all $1\le p<\infty$.

\begin{ex}\label{Ex:33}
Suppose that $1\le p<\infty$ and $p\neq 2$. 
\begin{enumerate}[(i)]
\item  Then there is a closed subspace $X\subset \ell^p$ that fails the $\QN_p$-AP.
\item Then there is a closed subspace $Z\subset c_0$ such that dim$\,\QA_Z^{\QN_p}=\infty$.
\end{enumerate}
\end{ex}
\begin{proof}
(i) By \cite[Corollary 1.1]{Reinov82} there are Banach spaces $E$ and $F$ together with an operator
\begin{equation}\label{222022}
T\in\QN_p(E,F)\setminus\overline{\F(E,F)}^{\V_{\varPi_p}}.
\end{equation}
It follows from \eqref{0311} that $T\notin \overline{\F(E,F)}^{\V_{\QN_p}}$.
According to the proof of \cite[Lemma 5]{PP69}, there are operators $A\in\K(E,Y)$, $S\in\QN_p(Y,X)$ and $B\in\K(X,F)$ such that $T=BSA$, where $X$ is a closed subspace of $\ell^p$. Since $T\notin\overline{\F(E,F)}^{\V_{\QN_p}}$ we have $S\notin\overline{\F(Y,X)}^{\V_{\QN_p}}$ by the operator ideal property, and thus, $X$ fails the $\QN_p$-AP. 
\smallskip

(ii) Consider the factorisation $T=BSA$ in part (i) of the quasi $p$-nuclear operator $T$ in \eqref{222022}. According to the compact factorisation result of Terzio\u{g}lu \cite{Terzioglu71} (or Randtke \cite{Randtke72}), the compact operators $A\in\K(E,Y)$ and $B\in\K(X,F)$ admit compact factorisations $A=A_2A_1$ and $B=B_2B_1$ through closed subspaces $M\subset c_0$ and $N\subset c_0$, respectively. Let $U:=B_1SA_2$ so that the following diagram commutes:
\begin{center}
\begin{tikzcd}
&E\arrow[dl,swap,"A_1"]\arrow{d}{A}\arrow{r}{T} &F&\\
M\arrow[bend right=20,swap]{rrr}{U}\arrow{r}{A_2} &Y\arrow{r}{S}&X\arrow{u}{B}\arrow{r}{B_1}&N\arrow[ul,swap,"B_2"].
\end{tikzcd} 
\end{center}
Here $U\in\QN_p(M,N)\setminus \overline{\F(M,N)}^{\V_{\QN_p}}$ by the operator ideal property, since $S\in\QN_p(Y,X)$ and $T=B_2UA_1\notin\overline{\F(E,F)}^{\V_{\QN_p}}$ by \eqref{222022}. Thus $\QN_p(M,N)\neq\overline{\F(M,N)}^{\V_{\QN_p}}$.

Finally, let $X_0=N$ and $X_n=M$ for all $n\in\mathbb N$ and consider the $c_0$-direct sum $Z:=\bigoplus_{n=0}^\infty X_n\subset c_0$. It then follows from Example \ref{infinitedimensional} that dim$\,\QA_Z^{\QN_p}=\infty$.
\end{proof}
For the analogous examples of closed subspaces failing the $\SK_p$-AP we recall the following surprising result of Oja \cite[Theorem 1]{Oja12}: If $X$ is a closed subspace of $L^p(\mu)$, then $X$ has the AP if and only if $X$ has the $\SK_p$-AP. 
\begin{ex}\label{Ex:34}
Suppose that $1\le p<\infty$ and $p\neq 2$.
\begin{enumerate}[(i)]
\item  Then there is a closed subspace $X\subset \ell^p$ that fails the $\SK_p$-AP.
\item Then there is a closed subspace $Z\subset c_0$ such that dim$\,\QA_Z^{\SK_p}=\infty$.
\end{enumerate}
\end{ex}
\begin{proof}
(i) Let $X\subset \ell^p$ be a closed subspace that fails the AP. Recall that such a subspace exists due to Davie \cite{Da73} for $p>2$ and Szankowski \cite{Sza78} for $1\le p<2$. Then $X$ fails the $\SK_p$-AP according to Oja's result \cite[Theorem 1]{Oja12} stated above.
\smallskip

(ii) Let $F$ be any Banach space that fails the $\SK_p$-AP. Then there is an operator $T\in \SK_p(E,F)\setminus\overline{\F(E,F)}^{\V_{\SK_p}}$ for a suitable Banach space $E$. According to \cite[Proposition 2.9]{GLT12}, there are Banach spaces $Y$ and $X$ together with operators $A\in\K(E,Y)$, $S \in\SK_p(Y,X)$ and $B\in\K(X,F)$ such that $T=BSA$. 

We next apply Terzio\u{g}lu's factorisation result \cite{Terzioglu71} on the compact operators $A\in\K(E,Y)$ and $B\in\K(X,F)$ and obtain, in a similar way as in Example \ref{Ex:33}.(ii), closed subspaces $M,N\subset c_0$ together with an operator $U\in\SK_p(M,N)\setminus\overline{\F(M,N)}^{\V_{\SK_p}}$. Let $Z:=\bigoplus_{n=0}^\infty X_n\subset c_0$, where $X_0=N$ and $X_n=M$ for all $n\in\mathbb N$. Then dim$\,\QA_Z^{\SK_p}=\infty$ by Example \ref{infinitedimensional}. 
\end{proof}

The following result exhibits a large class of Banach spaces that have the $\QN_p$-AP for every $2<p<\infty$. We refer to e.g. \cite[Section 6.2]{AK06} or \cite[Chapter 11]{DJT95} for the notions of type and cotype for Banach spaces. In particular, recall that $\ell^p$ and all closed subspaces $X\subset\ell^p$ have cotype 2 whenever $1\le p\le 2$ and type 2 whenever $2\le p<\infty$.

\begin{prop}\label{QNpcotype2}
Suppose that $X$ is a Banach space with cotype 2. Then $X$ has the $\QN_p$-AP for every $2<p<\infty$.
\end{prop}
\begin{proof}
Let $2<p<\infty$ and suppose that $T\in\QN_p(Y,X)$, where $Y$ is an arbitrary Banach space. According to the proof of \cite[Lemma 5]{PP69}, there is a closed subspace $M\subset\ell^p$ together with operators $A\in\QN_p(Y,M)$ and $B\in\K(M,X)$ such that $T=BA$. Now, since $M$ has type 2 and $X$ has cotype 2, we have $B\in\K(M,X)=\A(M,X)$ according to \cite[Theorem 2.2]{TW22}. It then follows from part (i) of Lemma \ref{minimal} below that $T=BA\in\overline{\F(Y,X)}^{\V_{\QN_p}}$. This concludes the proof.
\end{proof}
In the proof of Proposition \ref{QNpcotype2} we used the following elementary but useful lemma, which we state separately for convenient reference. 

\begin{lem}\label{minimal}
Let $\I=(\I,\VI)$ be a Banach operator ideal and let $X$, $Y$ and $Z$ be arbitrary Banach spaces.
\begin{enumerate}[(i)]
\item Suppose that $A\in \I(X,Y)$ and $B\in \A(Y,Z)$. Then $BA\in \overline{\F(X,Z)}^{\VI}$.
\item Suppose that $A\in \A(X,Y)$ and $B\in \I(Y,Z)$. Then $BA\in \overline{\F(X,Z)}^{\VI}$.
\end{enumerate}
\end{lem} 
\begin{proof} (i) Let $\varepsilon>0$ and pick a bounded finite-rank operator $F\in\F(Y,Z)$ such that $||B-F||<\varepsilon/||A||_\I$. Then $FA\in\F(X,Z)$ and 
\[
||BA-FA||_\I\le ||B-F||\,||A||_\I<\varepsilon
\]
by (BOI2). Consequently, $BA\in\overline{\F(X,Y)}^{\VI}$.
\smallskip
 
Part (ii) is proven similarly. 
\end{proof}

We proceed towards a result, which is similar to Proposition \ref{QNpcotype2}, for the Sinha-Karn $p$-compact operators in the case of $1\le p<2$.  
In the proof we use the following characterisation of the $\SK_p$-AP for dual Banach spaces, which is one of many equivalent conditions obtained by Delgado et al. \cite[Theorem 2.3]{DPS10_2}.
\begin{fact}\label{1612}
Suppose that $1\le p<\infty$ and let $X$ be a Banach space. Then the following are equivalent:
\begin{enumerate}[(i)]
\item $X^*$ has the $\SK_p$-AP.
\item $\QN_p(X,Y)=\overline{\F(X,Y)}^{\V_{\QN_p}}$ for every Banach space $Y$.
\end{enumerate}
\end{fact}
We also recall the following result of Persson and Pietsch \cite[Satz 43]{PP69}: 
\begin{equation}\label{1433}
||T||_{\QN_p}=||T||_{\varPi_p}
\end{equation}
for all $T\in\QN_p(X,Y)$, where $X$ and $Y$ are arbitrary Banach spaces (recall that in \cite{PP69} absolutely $p$-summing operators are called \emph{quasi-p-integral Abbildungen}). In particular, it follows from the Banach operator ideal inclusion $\QN_p\subset\varPi_p$ in \eqref{0311} and the identity \eqref{1433} that the approximative kernels of $\QN_p$ and $\varPi_p$ coincide, that is,
\begin{equation}\label{2012}
\overline{\F}^{\V_{\QN_p}}=\overline{\F}^{\V_{\varPi_p}}
\end{equation}
for all $1\le p<\infty$.

\begin{prop}\label{SKp}
Suppose that $X$ is a Banach space with cotype 2. Then $X^*$ has the $\SK_p$-AP for every $1\le p<2$.
\end{prop}

\begin{proof} 
Suppose that $1\le p<2$. By Fact \ref{1612} above it suffices to show that 
\begin{equation}\label{equ}
\QN_p(X,Y)=\overline{\F(X,Y)}^{||\cdot||_{\QN_p}}
\end{equation}
for every Banach space $Y$. For this, let $\varepsilon>0$ and suppose that $T\in\QN_p(X,Y)$, where $Y$ is an arbitrary Banach space. By the monotonicity \eqref{monot} of the classes $\QN_r$, we have $T\in\QN_2(X,Y)$. Moreover, since $\QN_2=\overline{\F}^{\V_{\QN_2}}$ (see Remarks \ref{9113}.(ii)), there is a bounded finite-rank operator $F\in\F(X,Y)$ such that $||F-T||_{\QN_2}<\varepsilon$. It follows from \eqref{0311} that $F-T\in\varPi_2(X,Y)$ and 
\begin{equation}\label{0311_4}
||F-T||_{\varPi_2}<\varepsilon.
\end{equation}
Next, since $X$ has cotype 2, a classical result of Maurey implies that $F-T\in\varPi_1(X,Y)$ and
\begin{equation}\label{0311_3}
||F-T||_{\varPi_1}\le M\cdot||F-T||_{\varPi_2},
\end{equation} 
where $M=M_X$ is a constant that depends only on the cotype 2 constant of $X$, see \cite[Theorem 5.16]{Pisier86} or \cite[Corollary 10.18]{TJ}. It follows that $F-T\in\varPi_p(X,Y)$ and 
\begin{equation}\label{1011}
||F-T||_{\varPi_p}\le||F-T||_{\varPi_1}
\end{equation}
 since $\varPi_1\subset\varPi_p$, see \cite[Theorem 2.8]{DJT95}. By combining the estimates  \eqref{1011}, \eqref{0311_3} and \eqref{0311_4} we obtain that 
\[
||F-T||_{\varPi_p}\le ||F-T||_{\varPi_1}\le M\cdot||F-T||_{\varPi_2}<M\cdot\varepsilon.
\]
Since the constant $M$ only depends on $X$, we have $T\in\overline{\F(X,Y)}^{\V_{\varPi_p}}=\overline{\F(X,Y)}^{\V_{\QN_p}}$, where the identity follows from \eqref{2012}. We have thus established \eqref{equ}, which then yields the claim.
\end{proof}
\begin{remark}
The claim in Proposition \ref{SKp} is also indicated on \cite[p. 280]{GLT12} for reflexive Banach spaces using a different approach that involves the Saphar's approximation property of order $p$ introduced in \cite{Saphar70}.
\end{remark}
We will require the following analogue of Fact \ref{1612} for the quasi $p$-nuclear operators. This follows from \cite[Theorem 4.6]{SK08}, since the approximation property of type $p$ introduced in \cite{SK08} is the same property as the $\QN_p$-AP (see \cite[p. 499]{DPS10_2}), and $(\SK_p^{dual})^{dual}=\SK_p$ by the dualities \eqref{911} and \eqref{9112}.

\begin{fact}\label{16122}
Suppose that $1\le p<\infty$ and let $X$ be an arbitrary Banach space. Then the following are equivalent:
\begin{enumerate}[(i)]
\item $X^*$ has the $\QN_p$-AP.
\item $\SK_p(X,Y)=\overline{\F(X,Y)}^{\V_{\SK_p}}$ for every Banach space $Y$.
\end{enumerate}
\end{fact}
In the interest of readability we indicate a proof of the implication (i)$\Rightarrow$(ii) with more details than the proof in \cite[Theorem 4.6]{SK08} provides.
\begin{proof}[Proof of (i)$\Rightarrow $(ii)] 
Assume that $X^*$ has the $\QN_p$-AP. Let $\varepsilon>0$ and suppose that $T\in \SK_p(X,Y)$, where $Y$ is an arbitrary Banach space. According to \cite[Theorem 3.1]{CK10}, there is a Banach space $Z$ together with operators $U\in\SK_p(X,Z)$ and $V\in\K(Z,Y)$ such that $T=VU$. By the DFJP-factorisation result (see e.g. \cite[Theorem 2.g.11]{LT79}), the (weakly) compact operator $V$ admits a factorisation $V=SR$ through a reflexive Banach space $W$, where $R\in\mathcal L(Z,W)$ and $S\in\mathcal L(W,Y)$ are bounded operators. By the operator ideal property, we have $RU\in\SK_p(X,W)$, and  thus, $(RU)^*\in\QN_p(W^*,X^*)$ by the duality \eqref{9112}. 

Next, since $X^*$ has the $\QN_p$-AP by assumption, there is a bounded finite-rank operator $F\in\F(W^*,X^*)$ such that $||(RU)^*-F||_{\QN_p}<\varepsilon/||S||$. Moreover, by the reflexivity of $W$ we have $RU=(RU)^{**}j$, where $j:X\to X^{**}$ denotes the canonical isometric embedding. Consequently, by applying (BOI2) and the duality \eqref{911}, we obtain that
\[
||RU-F^*j||_{\SK_p}=||(RU)^{**}j-F^*j||_{\SK_p}\le ||(RU)^{**}-F^*||_{\SK_p}=||(RU)^*-F||_{\QN_p}<\varepsilon/||S||.
\]
It follows that
\[||T-SF^*j||_{\SK_p}=||SRU-SF^*j||_{\SK_p}\le ||S||\,||RU-F^*j||_{\SK_p}<\varepsilon,\]
which shows that $T\in\overline{\F(X,Y)}^{\V_{\SK_p}}$.
\end{proof}
\begin{remark}
One can also establish the implication (i)$\Rightarrow$(ii) of Fact \ref{1612} in a similar way using the factorisation of quasi $p$-nuclear operators in \cite[Lemma 5]{PP69} and the DFJP-factorisation. This gives an alternative proof of the implication (1)$\Rightarrow$(5) in \cite[Theorem 2.3]{DPS10_2}.
\end{remark}
We are now in position to establish the main result of this section, which shows that the quotient algebras $\QAQN$ and $\QASK$ are trivial for large classes of Banach spaces. The case $p=2$ is obvious since $\QN_2$ and $\SK_2$ are approximative Banach operator ideals (see Remarks \ref{9113}.(ii)) and thus $\QA_X^{\QN_2}=\{0\}$ and $\QA_X^{\SK_2}=\{0\}$ for all Banach spaces $X$.
\begin{thm}\label{trivialquotient}
Let $X$ be a Banach space. \begin{enumerate}[(i)]
\item Suppose that $X$ has cotype 2. Then $\mathfrak A_X^{\QN_p}=\{0\}$ for all $1\le p<\infty$.
\item Suppose that $X^*$ has cotype 2. Then $\mathfrak A_X^{\SK_p}=\{0\}$ for all $1\le p<\infty$. 
\end{enumerate}
\end{thm}
\begin{proof}
(i) Suppose first that $1\leq p<2$. Then the dual space $X^*$ has the $\SK_p$-AP by Proposition \ref{SKp}. Thus $\mathfrak A_X^{\QN_p}=\{0\}$ by applying Fact \ref{1612}.
 
Next, suppose that $2< p<\infty$. Then $X$ has the $\QN_p$-AP by Proposition \ref{QNpcotype2} and consequently $\QAQN=\{0\}$.
\medskip

(ii) Suppose first that $1\le p<2$. Then $X^{**}$ has the $\SK_p$-AP by Proposition \ref{SKp}.  It follows that $X$ has the $\SK_p$-AP by \cite[Corollary 3.5]{DPS10_2} and thus $\QASK=\{0\}$.

Next, suppose that $2<p<\infty$. Then $X^*$ has the $\QN_p$-AP by Proposition \ref{QNpcotype2}. It follows from Fact \ref{16122} that $\QASK=\{0\}$.
\end{proof} 

Recall the long-standing open question whether $\K(X)=\A(X)$, that is, whether $\QA_X=\{0\}$ implies that the Banach space $X$ has the AP, see \cite[Problem 1.e.9]{LT77} or \cite[Problem 2.7]{Casazza01}. To conclude this section, we show that Theorem \ref{trivialquotient} together with the Banach spaces found in Examples \ref{Ex:33} and \ref{Ex:34} provides a negative answer to the analogous question for $\QN_p$ in the case of $1\leq p<2$ and for $\SK_p$ in the case of $2<p<\infty$. 

 \begin{ex}\label{trivialquotient2}
\begin{enumerate}[(i)]
\item Suppose that $1\le p<2$. Then there exists a Banach space $X$ that fails the $\QN_p$-AP for which $\QA_X^{\QN_p}=\{0\}$.
\item Suppose that $2< p<\infty$. Then there exists a Banach space $X$ that fails the $\SK_p$-AP for which $\QA_X^{\SK_p}=\{0\}$.
\end{enumerate}
\end{ex}
\begin{proof}
(i) By Example \ref{Ex:33}.(i) there is a closed subspace $X\subset\ell^p$ that fails the $\QN_p$-AP. Since $X$ has cotype 2, we have $\QA_{X}^{\QN_p}=\{0\}$ by Theorem \ref{trivialquotient}.(i).
\medskip

(ii) Let $X\subset\ell^p$ be a closed subspace that fails the $\SK_p$-AP, see Example \ref{Ex:34}.(i). Since $X$ has type 2, the dual space $X^*$ has cotype 2 according to \cite[Proposition 11.10]{DJT95}. Thus $\mathfrak A_X^{\SK_p}=\{0\}$ by Theorem \ref{trivialquotient}.(ii).
\end{proof}

\section{The index of nilpotency of the quotient algebras \texorpdfstring{$\QAQN$}{QAQNp} and \texorpdfstring{$\QASK$}{QASKp}}\label{section4}

An algebra $A$ is called \emph{nilpotent} if there is an integer $n\in\mathbb N$ such that 
\begin{equation}\label{nilpotent}
a_1\cdots a_n=0\text{ for all }a_1,\ldots,a_n\in A.
\end{equation} 
For a nilpotent algebra $A$ we call the smallest $n\in\mathbb N$ such that \eqref{nilpotent} holds the \emph{index} of $A$, denoted $index(A)$. 

In this section we discuss the nilpotency of the quotient algebras $\QAQN$ and $\QASK$. It follows from a multiplication rule due to Persson and Pietsch \cite{PP69} that these quotient algebras are nilpotent for all Banach spaces $X$. More precisely, the index of both $\QAQN$ and $\QASK$ is bounded from above by $\max\{2, \left\lceil p/2\right\rceil\}$, as we will show in Proposition \ref{prop41}. Here $\left\lceil x\right\rceil:=\min\{n\in\mathbb N\mid n\ge x\}$ for any positive real number $x>0$. In the main result of this section (Theorem \ref{mainsection4}) we show that this upper bound is exact in the sense that for every $p\neq 2$ there is a closed subspace $X\subset c_0$ such that 
\begin{equation}\label{identities}
index(\QAQN)=index(\QASK)=\max\{2,\left\lceil p/2\right\rceil\}.
\end{equation}
The key observation (Proposition \ref{prop43}) for the identities \eqref{identities} in the case $2<p<\infty$ is obtained by a modification of a sophisticated argument of Reinov \cite[Lemma 1.1]{Reinov82} which involves the matrix constructed by Davie \cite{Da75} related to the failure of the AP. Proposition \ref{prop43} is also applied in Section \ref{section5} when exhibiting non-trivial closed ideals of the quotient algebras $\QAQN$ and $\QASK$ as well as of the compact-by-approximable algebra $\QA_X$, for specific closed subspaces $X\subset\ell^p$ and $X\subset c_0$.
\smallskip

We recall the following multiplication rule for quasi $p$-nuclear operators \cite[Satz 48]{PP69}:
let $p,q,r\geq 1$ be such that $1/r=1/p+1/q$. Then 
\begin{equation}\label{multiplication}
VU\in\QN_r
\end{equation}
whenever $U\in\QN_p$ and $V\in\QN_q$ are compatible quasi $p$-nuclear operators.
By iterating the above formula we have
\begin{equation}\label{powermultiplication}
T_n \cdots T_2T_1\in\QN_{p/n}
\end{equation}
whenever $n\in\mathbb N$ is such that $1\le n\le p$ and $T_1,\ldots, T_n\in\QN_p$ are compatible operators.

By duality, the following analogue of \eqref{multiplication} holds for the Sinha-Karn $p$-compact operators: let $p,q,r\ge 1$ be such that $1/r=1/p+1/q$. Then
\begin{equation}\label{multip}
TS\in\SK_r
\end{equation}
whenever $S\in\SK_p$ and $T\in \SK_q$ are compatible Sinha-Karn $p$-compact operators. In fact, suppose that $S\in \SK_p(X,Y)$ and $T\in\SK_q(Y,Z)$ for arbitrary Banach spaces $X,Y$ and $Z$. By the duality \eqref{9112}, we have $S^*\in\QN_p(Y^*,X^*)$ and $T^*\in\QN_q(Z^*,Y^*)$, and consequently $(TS)^*=S^*T^*\in\QN_r(Z^*,X^*)$ by \eqref{multiplication}. Thus $TS\in \SK_r(X,Z)$ by \eqref{9112}. 

By iteration we have
\begin{equation}\label{powermultiSKp}
R_n\cdots R_2R_1\in\SK_{p/n}
\end{equation}
whenever $n\in\mathbb N$ is such that $1\le n\le p$ and $R_1,\ldots,R_n\in\SK_p$ are compatible operators.
\smallskip

We proceed with establishing that $\max\{2,\left\lceil\frac{p}{2}\right\rceil\}$ is an upper bound of the index of the quotient algebras $\QAQN$ and $\QASK$. Here again, the case of $p=2$ is obvious since $\QN_2=\overline{\F}^{\V_{\QN_2}}$ and $\SK_2=\overline{\F}^{\V_{\SK_2}}$ are approximative Banach operator ideals, see Remarks \ref{9113}.(ii). For the case of $1\le p<2$, we note further that
\begin{equation}\label{QNple2}
\QN_p\subset\QN_2\subset\A\quad\text{and}\quad\SK_p\subset\SK_2\subset\A
\end{equation}
for all $1\le p< 2$ by the monotonicity  \eqref{monot} and \eqref{monotonicity}.

\begin{prop}\label{prop41}
Suppose that $1\le p<\infty$ and let $\I=\QN_p$ or $\I=\SK_p$. Then 
\[index(\QA_X^\I)\le \max\{2,\left\lceil\frac{p}{2}\right\rceil\}\]
for any Banach space $X$.
\end{prop}

\begin{proof}
Let $X$ be an arbitrary Banach space and denote $m:=\max\{2,\left\lceil p/2\right\rceil\}$. We consider the cases $1\le p<2$ and $2<p<\infty$ separately. Note that $m=2$ if $1\le p<2$ and $m=\left\lceil p/2\right\rceil$ if $p>2$.

Suppose first that $1\le p< 2$ and let $S\in\I(X)$ and $T\in\I(X)$ be arbitrary. By \eqref{QNple2} we have $T\in\A(X)$ and consequently $TS\in\overline{\F(X)}^{\V_{\I}}$ by Lemma \ref{minimal}.(i). It follows that
\[(T+\FI)\cdot(S+\FI)=TS+\FI=0\in\QA_X^\I,\] 
and consequently $index(\QA_X^\I)\le 2=m$.

Next, suppose that $2<p<\infty$ and let $T_1,T_2,\ldots, T_{m}\in\QN_p(X)$ be arbitrary. By the multiplication rule \eqref{powermultiplication} we have $T_m\cdots T_2T_1\in\QN_{p/m}(X)$. Since $p/m\le 2$, the monotonicity \eqref{monot} yields
\[
T_m\cdots T_2T_1\in\QN_2(X).
\]
Moreover, since $\QN_2$ is approximative and $||\cdot||_{\QN_p}\le||\cdot||_{\QN_2}$, we have
\[
\QN_2(X)=\overline{\F(X)}^{\V_{\QN_2}}\subset\overline{\F(X)}^{\V_{\QN_p}},
\]
and thus $T_m\cdots T_2T_1\in\overline{\F(X)}^{\V_{\QN_p}}$. Then
\[\big(T_m+\overline{\F(X)}^{\V_{\QN_p}}\big)\cdots\big(T_1+\overline{\F(X)}^{\V_{\QN_p}}\big)=T_m\cdots T_2T_1+\overline{\F(X)}^{\V_{\QN_p}}=0\in\QAQN,\]
which shows that $index(\QA_X^{\QN_p})\le m$.

Using the multiplication rule \eqref{powermultiSKp} and the monotonicity \eqref{monotonicity} for the Sinha-Karn $p$-compact operators, a similar proof yields that $index(\QASK)\le m$. 
\end{proof}
\begin{remark}\label{141}
The fact that the quotient algebras $\QAQN$ and $\QASK$ are nilpotent for every Banach space $X$ is in sharp contrast with the current state of knowledge of the compact-by-approximable algebra $\QA_X=\K(X)/\A(X)$. In fact, it is not known whether there is a Banach space $X$ such that $\QA_X$ is non-trivial and nilpotent. However, we recall that there are known examples of Banach spaces $X$ such that  $\QA_X$ is non-nilpotent, see \cite[Proposition 3.1(i)]{TW21} and \cite[Theorem 2.9]{TW22}.
\end{remark}

We proceed towards the main result of this section (Theorem \ref{mainsection4}), which exhibits for each $p\neq 2$ a closed subspace $X\subset c_0$ for which the index of both $\QA_X^{\QN_p}$ and $\QASK$ is exactly the upper bound $\max\{2,\left\lceil p/2\right\rceil\}$ given in Proposition \ref{prop41}. The key observation for the case of $2<p<\infty$ is Proposition \ref{prop43} below, which is the most technical result of this paper. The proof is based on a factorisation argument \cite[Lemma 1.1]{Reinov82} due to Reinov. For another  modification of Reinov's argument of this type, see \cite[Theorem 3.9]{TW22}.

For the proof, recall from e.g. \cite[18.1 and 18.2]{Pietsch80} that the bounded operator $T\in\mathcal L(X,Y)$ is $p$-nuclear, denoted $T\in\mathcal N_p(X,Y)$, if there is a strongly $p$-summable sequence $(x_k^*)\in\ell^p_s(X^*)$ and a weakly $p'$-summable sequence $(y_k)\in \ell^{p'}_w(Y)$ such that $T=\sum_{k=1}^\infty x_k^*\otimes y_k$, that is,
\[
Tx=\sum_{k=1}^\infty x_k^*(x)y_k,\qquad x\in X.
\]
\begin{prop}\label{prop43}
Suppose that $2<p<\infty$ and let $m:=\left\lceil p/2\right\rceil$. Then there is a closed subspace $X\subset c_0$ together with operators $T\in\QN_p(X)$ and $U\in\SK_p(X)$ such that $T^{m-1}\notin \A(X)$ and $U^{m-1}\notin \A(X)$. 
\end{prop}
\begin{proof}
In order to simplify the notation in the argument we denote $n:=m-1$ so that
$2n<p\le 2(n+1)$. We consider the cases $2<p\le 4$ and $4<p<\infty$ separately. 

(i) Suppose first that $2<p\le 4$. Here the claim follows by a very similar argument as Example \ref{Ex:33}.(ii), since in this case $n=1$. In fact, in view of \cite[Lemma 1.1]{Reinov82} there is an operator 
\begin{equation*}
S\in\QN_p(Y,Z)\setminus\A(Y,Z)
\end{equation*}
for suitable reflexive Banach spaces $Y$ and $Z$. We then apply \cite[Lemma 5]{PP69} to obtain operators $A\in\K(Y,M)$, $\widehat R\in\QN_p(M,N)$ and $B\in\K(N,Z)$ such that $S=B\widehat RA$. By Terzio\v{g}lu's result \cite{Terzioglu71}, the compact operators $A$ and $B$ admit compact factorisations $A=A_2A_1$ and $B=B_2B_1$ through closed subspaces $M_1\subset c_0$ and $M_2\subset c_0$, respectively. Let $R:=B_1\widehat R A_2$. Since $\widehat R\in\QN_p(M,N)$ and $S=B_2RA_1\notin\A(Y,Z)$, we have
\[R\in\QN_p(M_1,M_2)\setminus \A(M_1,M_2)\] 
by the operator ideal property.

Next, by the duality \eqref{911} and reflexivity we have $S^*\in\SK_p(Z^*,Y^*)\setminus\A(Z^*,Y^*)$. According to \cite[Proposition 2.9]{GLT12}, there are operators $C\in\K(Z^*,M_3)$, $V\in\SK_p(M_3,M_4)$ and $D\in\K(M_4, Y^*)$ such that $S^*=DVC$. By a similar argument as above, we may assume (by applying \cite{Terzioglu71}) that $M_3$ and $M_4$ are closed subspaces of $c_0$. Observe that $V\notin\A(M_3,M_4)$ by the operator ideal property, since $S^*=DVC\notin\A(Z^*,Y^*)$.

Finally, consider the closed subspace $X:=M_1\oplus M_2\oplus M_3\oplus M_4\subset c_0$. Let $J_i:M_i\to X$ denote the natural isometric embedding and let $P_i:X\to M_i$ denote the natural projection for $i\in\{1,2,3,4\}$. By the operator ideal property we have that $J_2RP_1\in\QN_p(X)\setminus\A(X)$ and $J_4VP_3\in\SK_p(X)\setminus\A(X)$. This proves the claim for $2<p\le 4$. 
\smallskip

(ii) Suppose next that $4<p<\infty$ so that $n\ge 2$. Here we use a slightly modified version of the argument of Reinov in \cite[Lemma 1.1]{Reinov82}. The starting point is the infinite matrix $A=(a_{i,j})_{i,j=1}^\infty$ of scalars constructed by Davie \cite{Da73, Da75} (see also \cite[Theorem 2.d.3]{LT77}) with the following properties:
\begin{enumerate}[(i)]
\item $A^2=0$,
\item tr $A:=\sum_{k=1}^\infty a_{k,k}\neq 0$ and
\item $\sum_{k=1}^\infty \lambda_k^\alpha<\infty$ for all $\alpha>2/3$, where $\lambda_k:=\sup_{r\in\mathbb N}|a_{k,r}|>0$ for all $k\in\mathbb N$. 
\end{enumerate}
The matrix $A$ defines the following 1-nuclear operator
\[A:\ell^1\to\ell^1,\quad Ax=\big(\sum_{r=1}^\infty a_{k,r}x_r\big)_{k=1}^\infty,\quad x=(x_k)\in\ell^1\] 
for which properties (i) and (ii) hold. In fact, $A=\sum_{k=1}^\infty a_k\otimes e_k$, where $a_k:=(a_{k,r})_{r=1}^\infty\in\ell^\infty$ for all $k\in\mathbb N$ and $(e_k)$ denotes the unit vector basis of $\ell^1$. Since $(a_k)\in\ell^1_s(\ell^\infty)$ by (iii), we have $A\in\mathcal N_1(\ell^1)$.

We proceed by defining relevant spaces and operators, which are displayed in the commuting diagram \eqref{diagram} below. We start by defining the following bounded operators:
\begin{align*}
&V:\ell^1\to\ell^\infty,\quad Vx=(\lambda_k^{-1}\sum_{r=1}^\infty a_{k,r}x_r)_{k=1}^\infty,\quad x=(x_k)\in\ell^1,\\
&\Delta:\ell^\infty\to\ell^1,\quad \Delta x=(\lambda_k x_k)_{k=1}^\infty,\quad x=(x_k)\in\ell^\infty.
\end{align*}
Here $V$ is bounded (with $||V||\le 1$) since $|a_{k,r}|\le \lambda_k$ for all $k,r\in\mathbb N$, and $\Delta\in\mathcal N_1(\ell^\infty,\ell^1)$ since $(\lambda_k)\in\ell^1$ by (iii). Moreover, clearly $A=\Delta V$.

Next, recall that $p>2n$ since $n=\left\lceil p/2\right\rceil-1$ by definition. Thus we may pick a positive real number $a>0$ such that
\begin{equation}\label{aisbetween}
\frac{2}{3p}<a<\frac{1}{3n}.
\end{equation}
Define the following diagonal operator:
\[D:\ell^\infty\to c_0,\quad Dx=(\lambda_k^ax_k)_{k=1}^\infty,\quad x=(x_k)\in\ell^\infty.\]
Here $ap>2/3$ by \eqref{aisbetween} and thus $(\lambda_k^a)\in\ell^p$ by property (iii). This yields that $D\in\mathcal N_p(\ell^\infty,c_0)$. By applying \cite[Proposition 5.23]{DJT95}, we obtain operators $K\in \K(\ell^\infty,c_0)$ and $\Delta_0\in \mathcal N_p(c_0)$ such that $D=\Delta_0 K$.

For each $k\in\{1,\ldots,n-1\}$ let $\Delta_k=D|_{c_0}:c_0\to c_0$ denote the $p$-nuclear restriction of $D$. Also, define the following diagonal operator:
\[\Delta_n:c_0\to\ell^1, \quad \Delta_n x=(\lambda_k^{1-na}x_k)_{k=1}^\infty,\quad x=(x_k)\in c_0.\]
By \eqref{aisbetween} we have $1-na>2/3$. Thus $(\lambda_k^{1-na})\in\ell^1$ by (iii) and consequently $\Delta_n\in\mathcal N_1(c_0,\ell^1)$. Clearly $\Delta=\Delta_n\cdots\Delta_2\Delta_1 D$, since $1-na+(n-1)a+a=1$.

Next, let $V=\widetilde V Q$ be the canonical factorisation of the bounded operator $V\in\mathcal L(\ell^1,\ell^\infty)$ through the quotient space $E:=\ell^1/\ker V$. Define the closed subspace $Y_0:=\overline{KV\ell^1}\subset c_0$ and thereafter the operator \[S:E\to Y_0, \quad Sx=K\widetilde Vx,\quad x\in E.\]
This means that $j_0S=K\widetilde V$, where $j_0:Y_0\hookrightarrow c_0$ is the inclusion map. Consequently, $S\in\K(E,Y_0)$ since $K\widetilde V\in\K(E,c_0)$.

Proceed by defining successively for each $k\in\{0,\ldots,n-1\}$ the closed subspace $Y_{k+1}:=\overline{\Delta_{k}j_{k}Y_{k}}\subset c_0$, where $j_{k}:Y_{k}\hookrightarrow c_0$ is the inclusion map. Thereafter, define for each $k\in\{0,\ldots,n-1\}$ the operator 
\[\widetilde\Delta_{k}:Y_{k}\to Y_{k+1},\quad \widetilde\Delta_k y= \Delta_{k}j_{k}y,\quad y\in Y_k.\]
This means that $j_{k+1}\widetilde\Delta_k=\Delta_kj_k$ for each $k\in\{0,\ldots,n-1\}$, where $j_n:Y_n\hookrightarrow c_0$ also denotes the inclusion map for $k+1=n$. According to \cite[Satz 39]{PP69}, we have $\widetilde\Delta_k\in\QN_p(Y_k,Y_{k+1})$ for all $k\in\{0,\ldots,n-1\}$. 
 
The above defined spaces and operators are illustrated in the following commuting diagram:

\begin{equation}\label{diagram}
\begin{tikzcd}
\ell^{1} \arrow{r}{V} \arrow[d,swap,"Q"] &\ell^\infty\arrow{d}{K}\arrow{rd}{D} \arrow{rrrr}{\Delta} &&&&\ell^{1}\\
E\arrow{dr}{S}\arrow{ur}{\widetilde V} & c_0\arrow{r}{\Delta_0}& c_0 \arrow{r}{\Delta_1}&c_0\arrow{r}{\Delta_2}&\cdots\arrow{r}{\Delta_{n-1}}&c_0\arrow[swap]{u}{\Delta_n}\\
&Y_0\arrow[u,hook,swap,"j_0"]\arrow{r}{\widetilde{\Delta}_0}&Y_1\arrow{r}{\widetilde\Delta_1}\arrow[u,hook,swap,"j_1"]&Y_2\arrow{r}{\widetilde\Delta_2}\arrow[u,hook,swap,"j_2"]&\cdots\arrow{r}{\widetilde\Delta_{n-1}}&Y_n\arrow[u,hook,swap,"j_n"]\\
\end{tikzcd} 
\end{equation}
\textbf{Claim.}
We claim that
\begin{equation}\label{productnotapprox}
\widetilde\Delta_{n-1}\cdots\widetilde\Delta_1\widetilde\Delta_0S\notin \A(E,Y_n).
\end{equation}
Towards this, define
\[\Phi(U)=\text{tr}(\Delta_nj_nUQ)\]
for all $U\in\mathcal L(E,Y_n)$. The mapping $U\mapsto \Phi(U)$ defines a bounded linear functional $\Phi\in \mathcal L(E,Y_n)^*$. In fact, since $\Delta_n\in\mathcal N_1(c_0,\ell^1)$ we have $\Delta_nj_nUQ\in\mathcal N_1(\ell^1)$ for all $U\in\mathcal L(E,Y_n)$ by the operator ideal property. Moreover, since $\ell^1$ has the (metric) AP, the trace on the right hand side is well-defined and
\[|\Phi(U)|=|\text{tr}(\Delta_nj_nUQ)|\le ||\Delta_nj_nUQ||_{\mathcal N_1}\le ||\Delta_n||_{\mathcal N_1}||U||\]
for all $U\in\mathcal L(E,Y_n)$ (see e.g. \cite[10.3.2]{Pietsch80}), where the latter inequality holds by (BOI2). 

Observe that
\begin{equation}\label{notzero}
\Phi(\widetilde\Delta_{n-1}\cdots\widetilde\Delta_0S)=\text{tr}(\Delta_nj_n\widetilde\Delta_{n-1}\cdots\widetilde\Delta_0SQ)=\text{tr}(\Delta V)=\text{tr}(A)\neq 0
\end{equation}
by (ii). However, we will next show that $\Phi(U)=0$ for all $U\in\A(E,Y_n)$, which together with \eqref{notzero} yields the claim \eqref{productnotapprox}. 

Towards this, note that in view of linearity and continuity, it is sufficient to verify that $\Phi(x^*\otimes y)=0$ for all bounded rank-one operators $x^*\otimes y\in\F(E,Y_n)$, where $x^*\in E^*$ and $y\in Y_n$. Moreover, by construction we have
\[Y_n=\overline{\Delta_{n-1}\cdots\Delta_0KV\ell^1}=\overline{\widetilde\Delta_{n-1}\cdots\widetilde\Delta_0SQ\ell^1},\] 
and thus we may (by continuity) assume that $y\in \widetilde\Delta_{n-1}\cdots\widetilde\Delta_0SQ\ell^1$. This means that $y=\tilde\Delta_{n-1}\cdots \widetilde\Delta_0 SQx$ for some $x\in\ell^1$. Then
\begin{align}
\label{eq:ide}
\Phi(x^*\otimes y)=\text{tr}&(\Delta_nj_n(x^*\otimes y)Q)=\text{tr}(Q^*x^*\otimes \Delta_nj_ny)=x^*(Q\Delta_n j_n y)\\
\notag &=x^*(Q\Delta_nj_n\widetilde\Delta_{n-1}\cdots\widetilde\Delta_0SQx)=x^*(Q\Delta Vx)=0,
\end{align}
where the last equality holds since $Q\Delta V=0$. In fact, $\Delta \widetilde V Q\Delta V=(\Delta V)^2=0$ by (i) and the operator $\Delta\widetilde V$ is injective. Now, \eqref{eq:ide} together with \eqref{notzero} yields the claim \eqref{productnotapprox}.
\smallskip

We proceed towards the desired closed subspace $X\subset c_0$ and the desired operators $T\in\QN_p(X)$ and $U\in\SK_p(X)$. Note first for any fixed $k\in\{0,\ldots,n-1\}$ that $(\widetilde\Delta_k)^*\in\SK_p(Y_{k+1}^*,Y_k^*)$ by the duality \eqref{911}. Thus, according to \cite[Proposition 2.9]{GLT12}, there is a Banach space $Z_{k+1}$ together with operators $A_{k+1}\in\K(Y_{k+1}^*,Z_{k+1})$ and $B_{k+1}\in\SK_p(Z_{k+1},Y_k^*)$ such that $(\widetilde\Delta_k)^*=B_{k+1}A_{k+1}$. Here we can assume that the intermediate space $Z_{k+1}$ is a closed subspace of $c_0$ by applying Terzio\v{g}lu's factorisation result \cite{Terzioglu71} on the compact operator $A_{k+1}$, and the operator ideal property. Moreover, since $S\in\K(E,Y_0)$, also the compact adjoint operator $S^*\in\K(Y_0^*,E^*)$ admits a compact factorisation $S^*=B_0A_0$ through a closed subspace $Z_0\subset c_0$ by \cite{Terzioglu71}. 

For all $k\in\{0,\ldots,n-1\}$ define the operator $R_k:=A_kB_{k+1}$. Then the operators $R_k\in\SK_p(Z_{k+1},Z_k)$ by the operator ideal property, and the following diagram commutes:
\begin{center}
\begin{tikzcd}
Y_n^* \arrow{rr}{(\widetilde\Delta_{n-1})^*} \arrow[dr,"A_{n}"] &&Y_{n-1}^*\arrow{dr}{A_{n-1}} \arrow{rr}{(\widetilde\Delta_{n-2})^*} &&\cdots \arrow{r}{(\widetilde\Delta_1)^*}&Y_1^*\arrow{rr}{(\widetilde\Delta_0)^*}\arrow[dr,swap,"A_1"]&&Y_0^*\arrow{r}{S^*}\arrow[dr,"A_0"]&E^*\\
&Z_n\arrow[ur,swap,"B_{n}"]\arrow[rr,swap,"R_{n-1}"] &&Z_{n-1}\arrow[rr,swap,"R_{n-2}"]&&\cdots \arrow[r,swap,"R_1"]&Z_1\arrow[ur,swap,"B_1"]\arrow[rr,swap,"R_0"]&&Z_0\arrow[swap]{u}{B_0}
\end{tikzcd}
\end{center}
Finally, consider $X:=(Y_0\oplus\cdots\oplus Y_n)\oplus (Z_0\oplus\cdots\oplus Z_n)\subset c_0$ and define the operator
\[T:X\to X,\quad T\big((y_0,\ldots,y_n)+(z_0,\ldots,z_n)\big)=(0,\widetilde\Delta_0 y_0,\ldots,\widetilde\Delta_{n-1}y_{n-1})+(0,\ldots,0).\]
This means that $T=\sum_{k=0}^{n-1}J_{Y_{k+1}}\widetilde\Delta_kP_{Y_k}$, where $P_{Y_k}:X\to Y_k$ is the canonical projection and $J_{Y_k}:Y_{k}\to X$ is the natural isometric embedding for all $k\in\{0,\ldots,n\}$. We also define the operator
\[U:X\to X,\quad U\big((y_0,\ldots,y_n)+(z_0,\ldots,z_n)\big)=(0,\ldots,0)+(R_0z_1,\ldots,R_{n-1}z_n,0),\]
which means that $U=\sum_{k=0}^{n-1}J_{Z_k}R_kP_{Z_{k+1}}$, where $P_{Z_k}:X\to Z_k$ is the canonical projection and $J_{Z_k}:Z_k\to X$ is the natural isometric embedding for all $k\in\{0,\ldots,n\}$. 
\smallskip

\textbf{Claims.} 
\begin{enumerate}
\item[(a)] $T\in\QN_p(X)$ and $T^n\notin\A(X)$.
\item[(b)] $U\in\SK_p(X)$ and $U^n\notin\A(X)$.
\end{enumerate}
For the claims in part (a), observe first that $T\in\QN_p(X)$ by the operator ideal property since $\widetilde\Delta_k$ is quasi $p$-nuclear for every $k\in\{0,\ldots,n-1\}$. Next, note that 
\[T^n\big((y_0,\ldots,y_n)+(z_0,\ldots,z_n)\big)=(0,\ldots,0,\widetilde\Delta_{n-1}\cdots\widetilde\Delta_0 y_0)+(0,\ldots, 0,0),\]
which means that $T^{n}=J_{Y_n}\widetilde\Delta_{n-1}\cdots\widetilde\Delta_0 P_{Y_0}$. Now, since \[P_{Y_n}T^nJ_{Y_0}S=P_{Y_n}J_{Y_n}\tilde\Delta_{n-1}\cdots\widetilde\Delta_0P_{Y_0}J_{Y_0} S=\tilde\Delta_{n-1}\cdots\widetilde\Delta_0 S\notin\A(E,Y_n)\] by \eqref{productnotapprox}, we have $T^{n}\notin\A(X)$ by the operator ideal property. 

For the claims in part (b), observe that $U\in\SK_p(X)$ since $R_k$ is a Sinha-Karn $p$-compact operator for every $k\in\{0,\ldots,n-1\}$. Finally, we verify that $U^n\notin\A(X)$. 

For this, note that
\[U^n\big((y_0,\ldots,y_n)+(z_0,\ldots,z_n)\big)=(0,0,\ldots,0)+(R_{0}\cdots R_{n-1}z_n,0,\ldots,0),\]
which means that $U^n=J_{Z_0}R_0\cdots R_{n-1}P_{Z_n}$. Thus
\[
B_0P_{Z_0}U^nJ_{Z_n}A_n= B_0R_0\cdots R_{n-1}A_n= S^*(\widetilde\Delta_0)^*\cdots (\widetilde\Delta_{n-1})^*=(\widetilde\Delta_{n-1}\cdots\widetilde\Delta_0 S)^*
.\]
Now, if $U^n\in\A(X)$, then $(\widetilde\Delta_{n-1}\cdots\widetilde\Delta_0 S)^*\in\A(Y_n^*,E^*)$ by the above identity and the operator ideal property. But this would mean that $\widetilde\Delta_{n-1}\cdots\widetilde\Delta_0 S\in\A(E,Y_n)$ (see  e.g. \cite[Theorem 11.7.4]{Pietsch80}) which contradicts  \eqref{productnotapprox}. Consequently, $U^n\notin\A(X)$.
\end{proof}
After these preparations we are now in position to establish the main result of this section.
\begin{thm}\label{mainsection4}
Suppose that $1\le p<\infty$ and $p\neq 2$. Then there is a closed subspace $X\subset c_0$ such that 
\[index(\QAQN)=index(\QASK)= \max\{2,\left\lceil \frac{p}{2}\right\rceil\}.
\]
\end{thm}
\begin{proof}
We consider the cases $1\le p<2$ and $2<p<\infty$ separately. Denote $m:=\max\{2,\left\lceil \frac{p}{2}\right\rceil\}$ so that $m=2$ if $1\le p<2$ and $m=\lceil p/2\rceil$ if $2<p<\infty$.

Suppose first that $1\le p<2$. By Example \ref{Ex:33}.(ii) and Example \ref{Ex:34}.(ii) there are closed subspaces $Z_1,Z_2\subset c_0$ together with operators $T\in\QN_p(Z_1)\setminus\overline{\F(Z_1)}^{\V_{\QN_p}}$ and $U\in\SK_p(Z_2)\setminus\overline{\F(Z_2)}^{\V_{\SK_p}}$. Consider the direct sum $X:=Z_1\oplus Z_2\subset c_0$. It is then easy to check that $\QAQN\neq\{0\}$ and $\QASK\neq\{0\}$. In fact, $J_1TP_1\in\QN_p(X)\setminus\overline{\F(X)}^{\V_{\QN_p}}$ by the operator ideal property, where $J_1:Z_1\to X$ is the natural isometric embedding and $P_1:X\to Z_1$ is the natural projection. Similarly $J_2UP_2\in\SK_p(X)\setminus\overline{\F(X)}^{\V_{\SK_p}}$. Consequently, the index of both $\QAQN$ and $\QASK$ is at least 2. It then follows from Proposition \ref{prop41} that $index(\QAQN)=index(\QASK)=2$.
\smallskip

Next, suppose that $2<p<\infty$. By Proposition \ref{prop43} there is a closed subspace $X\subset c_0$ together with operators $T\in \QN_p(X)$ and $U\in\SK_p(X)$ such that 
$T^{m-1}\notin\A(X)$ and $U^{m-1}\notin\A(X)$. By (BOI3) we have $T^{m-1}\notin\overline{\F(X)}^{\V_{\QN_p}}$ and thus
\[\big(T+\overline{\F(X)}^{\V_{\QN_p}}\big)^{m-1}=T^{m-1}+\overline{\F(X)}^{\V_{\QN_p}}\neq 0\in\QAQN.\]
Similarly, $U^{m-1}\notin\overline{\F(X)}^{\V_{\SK_p}}$ and thus $\big(U+\overline{\F(X)}^{\V_{\SK_p}}\big)^{m-1}\neq 0\in\QASK$. Consequently, the index of both $\QAQN$ and $\QASK$ is at least $m$, and thus by Proposition \ref{prop41} we have $index(\QAQN)=index(\QASK)=m$.
\end{proof}

\section{Closed ideals}\label{section5}

In this section we find examples of Banach spaces $X$ for which the quotient algebras $\QAQN$, $\QASK$ as well as $\QA_X=\K(X)/\A(X)$ contain chains of non-trivial closed ideals of a natural type in the case of $2<p<\infty$. We recall here that by a closed ideal of a given Banach algebra $A=(A,\V_A)$ we always mean a $\V_A$-closed two-sided ideal of $A$. Our results draw on Proposition \ref{prop43} where the existence of operators $T\in\QN_p$ and $U\in\SK_p$ was established for which $T^{\left\lceil p/2\right\rceil-1}\notin\A$ and $U^{\left\lceil p/2\right\rceil-1}\notin\A$, as well as the multiplication rules \eqref{multiplication} and \eqref{multip} for the classes $\QN_r$ and $\SK_r$, respectively. In the main theorem (Theorem \ref{2083}) of this section we exhibit a closed subspace $X\subset c_0$ such that the quotient algebra $\QA_X$ carries two incomparable countably infinite chains of nilpotent closed ideals. Here the closed ideals of $\QA_X$ are induced from ideals of the form $\QN_s(X)^r$ and $\SK_s(X)^r$ for carefully chosen natural numbers $s,r\in\mathbb N$. Moreover, these ideals also induce an infinite chain of closed ideals of the quotient algebras $\QAQN$ and $\QASK$, respectively.

We start by reviewing a basic fact about ideals of operators which we will require. This follows from e.g. \cite[Theorem 2.5.8(i)]{Dales00} but we provide the proof for completeness. 
\begin{fact}\label{090909}
Let $X$ be a Banach space and let $\I=(\I,\VI)$ be an arbitrary Banach operator ideal. If $\J$ is a non-zero algebraic ideal of $\I(X)$, then $\F(X)\subset\J$. In particular, if $\J$ is a non-zero closed ideal of $\I(X)$, then $\overline{\F(X)}^{\VI}\subset\J$.
\end{fact}
\begin{proof}
Suppose that $\mathcal J$ is a non-zero algebraic ideal of $\I(X)$ and let $S\in \F(X)$ be a bounded finite-rank operator. By linearity we may assume that $S=y^*\otimes y$ for some  $y^*\in X^*$ and $y\in X$.
Since $\J$ is non-zero by assumption, there is a non-zero operator $T\in\J$. Pick $x^*\in X^*$ and $x\in X$ such that $x^*(Tx)=1$. It is then straightforward to check that \[S=y^*\otimes y=(x^*\otimes y)T(y^*\otimes x).\] This shows that $S\in\J$ since $\I(X)$ contains the bounded rank-one operators by (BOI1) and $\J$ is an ideal of $\I(X)$ by assumption. Consequently, $\mathcal F(X)\subset\mathcal J$.
\end{proof}
Let $q:\mathcal I(X)\to \mathfrak A_X^\I$ denote the quotient map for an arbitrary Banach operator ideal $\I=(\I,\VI)$ and Banach space $X$. We proceed by observing that closed ideals of the quotient algebra $\QA_X^\I$ correspond to non-zero closed ideals of $\I(X)$. Here the case of $\I=\K$ was provided in \cite[Proposition 2.1]{TW22}.
\begin{prop}\label{quotient}
Let $X$ be a Banach space and let $\I=(\I,\VI)$ be an arbitrary Banach operator ideal.
\begin{enumerate}[(i)]
\item Suppose that $\J$ is a non-zero closed ideal of $\I(X)$. Then $q(\J)$ is a closed ideal of $\QA_X^\I$. 
\item Suppose that $\widehat \J$ is a closed ideal of $\QA_X^\I$. Then $q^{-1}(\widehat\J)$ is a non-zero closed ideal of $\I(X)$.
\item Suppose that $\J_1$ and $\J_2$ are two non-zero closed ideals of $\I(Z)$. Then 
\[\J_1\subset\J_2\text{ if and only if }q(\J_1)\subset q(\J_2).\]
\end{enumerate}
In particular, the rule $\J\mapsto q(\J)$ defines a bijective correspondence between  the set of non-zero closed ideals of $\I(X)$  and the set of closed ideals of $\QA_X^\I$, which preserves inclusions in both directions.
\end{prop}
\begin{proof}
(i) Since $q$ is a surjective algebra homomorphism, the image $q(\J)$ is an ideal of $\QA_X^\I$. For the claim that $q(\J)$ is closed in $\QA_X^\I$, suppose that $T\in\I(X)$ is an operator such that $q(T)\in\overline{q(\J)}$ and let $\varepsilon>0$. Then there is an operator $S\in\J$ such that $||q(T-S)||<\varepsilon$. By definition there is an operator $R\in\overline{\F(X)}^{\VI}$ such that $||T-S-R||_{\I}<\varepsilon$. By Fact \ref{090909} we have $R\in\J$ since $\J=\overline{\J}^{\VI}$, and thus $S-R\in\J$. Consequently, $T\in\overline{\J}^{\VI}=\J$, which means that $q(T)\in q(\J)$.
\smallskip

(ii) Clearly $q^{-1}(\widehat \J)$ is a closed ideal of $\I(X)$ since $q$ is a bounded algebra homomorphism. Moreover, $q^{-1}(\widehat\J)\neq\{0\}$ since $\ker q\neq\{0\}$ (assuming tacitly that $X\neq \{0\}$).
\smallskip

(iii) The forward implication is clearly true. For the converse implication, suppose that $q(\J_1)\subset q(\J_2)$. It follows that $q^{-1}q(\J_1)\subset q^{-1}q(\J_2)$. Thus $\J_1\subset\J_2$, since \begin{equation}\label{idealsid}
\J=q^{-1}q(\J)\end{equation} 
for any non-zero closed ideal $\J$ of $\I(X)$. In fact, suppose that $\J$ is a non-zero closed ideal of $\I(X)$. In order to verify the non-trivial inclusion $q^{-1}q(\J)\subset\J$ of \eqref{idealsid}, let $T\in q^{-1}q(\J)$ so that $q(T)\in q(\J)$. Thus $q(T-S)=0$ for some $S\in\J$. Now, $\ker q =\overline{\F(X)}^{\VI}\subset\J$ by Fact \ref{090909} and consequently $T-S\in\J$, which then yields $T\in\J$. 
\end{proof}

According to \cite[Theorem 3.12]{TW22}, there is for each $4<p<\infty$ a closed subspace $X\subset\ell^p$ such that the uniform closure $\overline{\QN_p(X)}$ lies strictly between $\A(X)$ and $\K(X)$. By duality, $\overline{\SK_p(X^*)}$ lies strictly between $\A(X^*)$ and $\K(X^*)$. Our first theorem on closed ideals expands on this in the case of $p>6$ by exhibiting a closed subspace $X\subset \ell^p$ that carries a finite chain of closed ideals of $\mathcal L(X)$ that lie strictly between $\A(X)$ and $\overline{\QN_p(X)}$, where the length of the chain depends on the parameter $p$. By duality, we obtain a finite chain of closed ideals strictly between $\A(X^*)$ and $\overline{\SK_p(X^*)}$.

We will require the following lemma which follows from Proposition \ref{prop41}. 
\begin{lem}\label{16120}
Suppose that $2<p<\infty$ and let $\I=\QN_p$ or $\I=\SK_p$. Then 

\[index(\overline{\I(X)}/\A(X))\le \left\lceil p/2 \right\rceil.\]
In particular, the quotient algebra $\overline{\I(X)}/\A(X)$ is nilpotent.
\end{lem}
\begin{proof}
Suppose that $T\in\overline{\I(X)}$ and denote $m:=\left\lceil p/2\right\rceil$. Let $(S_n)\subset\I(X)$ be a sequence such that $||T-S_n||\to 0$ as $n\to\infty$. Then $||T^m-S_n^m||\to 0$ as $n\to\infty$. By Proposition \ref{prop41} we have $S_n^m\in\overline{\F(X)}^{\VI}$ for every $n\in\mathbb N$, and thus $T^m\in\A(X)$ in view of the fact that $\overline{\F(X)}^{\VI}\subset\A(X)$. This yields the claim.
\end{proof}
For any Banach operator ideal $\I=(\I,\VI)$ and any Banach space $X$ we denote \[\I(X)^n:=\text{span}\{T_n\cdots T_2T_1\mid T_1,\ldots,T_n\in \I(X)\},\quad n\in\mathbb N.\]
The operator ideal property implies that the family $\{\I(X)^n\mid n\in\mathbb N\}$ forms a decreasing chain of algebraic ideals of $\mathcal L(X)$. We are now in position to prove our first result on closed ideals of operators.

\begin{thm}\label{0301}
Suppose that $6<p<\infty$. Denote $m:=\left\lceil p/2\right\rceil$ so that $2(m-1)<p\le 2m$. Then there is a closed subspace $Z\subset\ell^p$ for which
\begin{equation}\label{0901}
\A(Z)\subsetneq \overline{\QN_p(Z)^{m-2}}\subsetneq\cdots\subsetneq\overline{\QN_p(Z)^2}\subsetneq \overline{\QN_p(Z)}\subsetneq\K(Z),
\end{equation}
\begin{equation}\label{0901_a}
\A(Z^*)\subsetneq \overline{\SK_p(Z^*)^{m-2}}\subsetneq\cdots\subsetneq\overline{\SK_p(Z^*)^2}\subsetneq \overline{\SK_p(Z^*)}\subsetneq\K(Z^*).
\end{equation}
Moreover, the following hold:
\begin{equation}\label{09012}
\overline{\F(Z)}^{||\cdot||_{\QN_p}}\subsetneq \overline{\QN_p(Z)^{m-2}}^{||\cdot||_{\QN_p}}\subsetneq\cdots\subsetneq\overline{\QN_p(Z)^2}^{||\cdot||_{\QN_p}}\subsetneq\QN_p(Z),
\end{equation}
\begin{equation}\label{09012_a}
\overline{\F(Z^*)}^{||\cdot||_{\SK_p}}\subsetneq \overline{\SK_p(Z^*)^{m-2}}^{||\cdot||_{\SK_p}}\subsetneq\cdots\subsetneq\overline{\SK_p(Z^*)^2}^{||\cdot||_{\SK_p}}\subsetneq\SK_p(Z^*).
\end{equation}
Thus, by Proposition \ref{quotient} the compact-by-approximable algebras $\QA_Z=\K(Z)/\A(Z)$ and $\QA_{Z^*}=\K(Z^*)/\A(Z^*)$ both contain a chain of $m-2$ non-trivial closed ideals which are nilpotent by Lemma \ref{16120}. Furthermore, the quotient algebras $\QA_Z^{\QN_p}$ and $\QA_{Z^*}^{\SK_p}$ both contain a chain of $m-3$ non-trivial closed ideals.
\end{thm} 
\begin{proof}
By Proposition \ref{prop43} there is a closed subspace $X\subset c_0$ together with an operator $T\in\QN_p(X)$ such that \begin{equation}\label{167_3}
T^{m-1}\notin\A(X).\end{equation} 
According to the proof of \cite[Lemma 5]{PP69}, the operator $T\in\QN_p(X)$ admits a factorisation 
\begin{equation}\label{241bv}
T=VU
\end{equation} 
where $U\in\QN_p(X,Z_0)$, $V\in\K(Z_0,X)$ and $Z_0$ is a closed subspace of $\ell^p$. Moreover, by \cite[Theorem 2.9]{TW22} there is a closed subspace $Z_1\subset\ell^p$ together with an operator $\widetilde R\in\K(Z_1)$ such that 
\begin{equation}\label{157}
\widetilde R^{n}\notin\A(Z_1)
\end{equation} 
for all $n\in\mathbb N$. Consider the direct sum $Z:=Z_0\oplus Z_1\subset\ell^p$. 

We will verify in Steps 1--3 below that the closed subspace $Z\subset\ell^p$ satisfies the claims \eqref{0901}--\eqref{09012_a}. For Steps 1 and 2, let $P_0:Z\to Z_0$ and $P_1:Z\to Z_1$ denote the natural projections and let $J_0:Z_0\to Z$ and $J_1:Z_1\to Z$ be the corresponding isometric embeddings.
\medskip

\textbf{Step 1.} \emph{Claim.} $\overline{\QN_p(Z)}\subsetneq\K(Z)$.
\smallskip

Let $R:=J_1\widetilde RP_1\in\K(Z)$. Since $\widetilde R^n=P_1R^nJ_1\notin\A(X)$ for all $n\in\mathbb N$, the operator ideal property implies that $R^n\notin\A(X)$ for all $n\in\mathbb N$. Consequently, $\QA_X=\K(X)/\A(X)$ is non-nilpotent. However, the quotient $\overline{\QN_p(Z)}/\A(Z)$ is nilpotent by Lemma \ref{16120} and thus $\overline{\QN_p(Z)}\subsetneq\K(Z)$.
\medskip

\textbf{Step 2.} \emph{Claim.} For every $k\in\{1,\ldots,m-3\}$ the following hold: 
\begin{enumerate}
\item[(a)] $\overline{\QN_p(Z)^{k+1}}\subsetneq \overline{\QN_p(Z)^k}$.  
\item[(b)] $\overline{\QN_p(Z)^{k+1}}^{\V_{\QN_p}}\subsetneq \overline{\QN_p(Z)^{k}}^{\V_{\QN_p}}$.
\end{enumerate}
Let $k\in\{1,\ldots,m-3\}$ be fixed. For part (a) we first note that $\QN_p(Z)^{k+1}\subset\QN_p(Z)^k$ by the operator ideal property, and thus $\overline{\QN_p(Z)^{k+1}}\subset \overline{\QN_p(Z)^k}$. In order to show the strict inclusion, recall that $U\in\QN_p(X,Z_0)$ in \eqref{241bv} above. Thus \begin{equation}\label{08014}
S:=J_0UVP_0\in\QN_p(Z)
\end{equation}
by the operator ideal property. We claim that
 \begin{equation}\label{0801}
 S^k\in \QN_p(Z)^k\setminus \overline{\QN_p(Z)^{k+1}},
 \end{equation} 
which then yields the claim in part (a).
 
For this, note first that $S^k\in\QN_p(Z)^k$. Next, we assume, towards a contradiction, that $S^k\in\overline{\QN_p(Z)^{k+1}}$. Thus there is a sequence $(R_n)\subset\QN_p(Z)^{k+1}$ such that $||R_n-S^k||\to 0$ as $n\to\infty$. Since 
\begin{equation}\label{080122}
T^{m-1}=(VU)^{m-1}=V(UV)^{m-2}U=VP_0(J_0UVP_0)^{m-2}J_0U=VP_0S^{m-2}J_0U\end{equation} we have that
\begin{align}\label{te}
||VP_0R_n S^{m-k-2}&J_0U-T^{m-1}||=||VP_0(R_n-S^k)S^{m-k-2}J_0U||\\
\notag&\leq ||V||\,||S^{m-k-2}||\,||U||\,||R_n-S^k||\to 0
\end{align}
as $n\to\infty$. By linearity and the multiplication rule \eqref{powermultiplication} we have $S^{m-k-2}\in\QN_{p/(m-k-2)}(Z)$ and $R_n\in\QN_{p/(k+1)}(Z)$ for all $n\in\mathbb N$. Moreover, since $U\in\QN_p(X,Z_0)$, the multiplication rule \eqref{multiplication} implies that 
\[VP_0R_nS^{m-k-2}J_0U\in\QN_q(X)\]
for all $n\in\mathbb N$, where
\[\frac{1}{q}=\frac{1}{p}+\frac{m-k-2}{p}+\frac{k+1}{p}=\frac{m}{p}\geq \frac{m}{2m}=\frac{1}{2}.\]
Hence $q\le 2$ which yields that $VP_0R_nS^{m-k-2}J_0U\in\A(X)$ for all $n\in\mathbb N$ by \eqref{QNple2}. But then \eqref{te} implies that $T^{m-1}\in\A(Z)$, which contradicts \eqref{167_3}. Thus the claim \eqref{0801} holds, which yields part (a). 

Part (b) follows immediately from part (a) since $\V\le \V_{\QN_p}$ and $\QN_p(Z)^{k+1}\subset\QN_p(Z)^k$ for all $k$. In fact, if $\overline{\QN_p(Z)^{k+1}}^{\V_{\QN_p}}=\overline{\QN_p(Z)^k}^{\V_{\QN_p}}$ for some $k\in\{1,\ldots,m-3\}$, then $\overline{\QN_p(Z)^{k+1}}=\overline{\QN_p(Z)^k}$, which contradicts the claim in part (a).
\medskip

\textbf{Step 3.} \emph{Claim.} 
\begin{enumerate}[(a)] 
\item $\A(Z)\subsetneq\overline{\QN_p(Z)^{m-2}}$. 
\item $\overline{\F(Z)}^{\V_{\QN_p}}\subsetneq \overline{\QN_p(Z)^{m-2}}^{\V_{\QN_p}}$.
\end{enumerate}

Let $S\in\QN_p(Z)$ be the operator defined in \eqref{08014}. We observe that \[S^{m-2}\in\QN_p(Z)^{m-2}\setminus\A(Z)\] which yields the claim in part (a). In fact, assume, towards a contradiction, that
 $S^{m-2}\in\A(Z)$. Since $T^{m-1}=VP_0S^{m-2}J_0U$ by \eqref{080122}, we have $T^{m-1}\in\A(Z)$ by the operator ideal property. But this contradicts \eqref{167_3}.
 
 Part (b) follows again directly from part (a) since $\V\le\V_{\QN_p}$.
 \smallskip
 
It follows from Steps 1--3 that the claims in \eqref{0901} and \eqref{09012} hold. Since $Z$ is reflexive it is then straightforward to verify that the claims in \eqref{0901_a} and \eqref{09012_a} hold by applying the duality \eqref{911} and Schauder's theorem. In fact, by \eqref{911} and the reflexivity of $Z$, the following hold for all $T\in\mathcal L(Z)$ and all $k\in\{1,\ldots,m-2\}$: $T\in\overline{\QN_p(Z)^k}$ if and only if $T^*\in\overline{\SK_p(Z^*)^k}$, and similarly, $T\in\overline{\QN_p(Z)^k}^{\V_{\QN_p}}$ if and only if $T^*\in\overline{\SK_p(Z^*)^k}^{\V_{\SK_p}}$. 
\end{proof} 
\begin{remarks}
Let $6<p<\infty$ and denote $m:=\lceil p/2\rceil$. 
\smallskip

(i) We do not know whether there exists a closed subspace $Z\subset\ell^p$ so that
\begin{equation*}\A(Z)\subsetneq \overline{\QN_p(Z)^{m-1}}\subsetneq\overline{\QN_p(Z)^{m-2}}\subsetneq \cdots\subsetneq\overline{\QN_p(Z)^2}\subsetneq \overline{\QN_p(Z)}\subsetneq\K(Z).\end{equation*}
Such an example would improve upon Theorem \ref{0301} and the chain $\big\{\overline{\QN_p(Z)^k}\big\}_{k=1}^{m-1}$ of non-trivial closed ideals of $\mathcal K(Z)$ would have maximal length in the sense that  $\A(Z)=\overline{\QN_p(Z)^{n}}$ for every $n\ge m$, which holds by Proposition \ref{prop41}.
\smallskip

(ii) By using a similar argument as in Theorem \ref{0301} one obtains a closed subspace $Z\subset c_0$ together with the following chains of closed ideals with maximal length (in the sense described above):
\begin{equation}\label{3o}
\A(Z)\subsetneq \overline{\QN_p(Z)^{m-1}}\subsetneq\overline{\QN_p(Z)^{m-2}}\subsetneq\cdots\subsetneq\overline{\QN_p(Z)^2}\subsetneq \overline{\QN_p(Z)}\subsetneq\K(Z),
\end{equation}
\begin{equation}\label{4k}
\A(Z)\subsetneq \overline{\SK_p(Z)^{m-1}}\subsetneq\overline{\SK_p(Z)^{m-2}}\subsetneq\cdots\subsetneq\overline{\SK_p(Z)^2}\subsetneq \overline{\SK_p(Z)}\subsetneq\K(Z).
\end{equation}
In fact, by Proposition \ref{prop43} there is a closed subspace $Z_0\subset c_0$ together with operators $T\in \QN_p(Z_0)$ and $U\in\SK_p(Z_0)$ for which $T^{m-1}\notin\A(Z_0)$ and $U^{m-1}\notin\A(Z_0)$. Moreover, according to \cite[Theorem 2.9]{TW22}, there is a closed subspace $Z_1\subset c_0$ together with an operator $S\in\K(Z_1)$ such that $S^n\notin\A(Z_1)$ for every $n\in\mathbb N$. One then verifies that the direct sum $Z=Z_0\oplus Z_1\subset c_0$ has the desired properties \eqref{3o} and \eqref{4k} by verifying the analogues of Steps 1--3 in Theorem \ref{0301} for the classes $\QN_p(Z)$ and $\SK_p(Z)$. We leave the details to the interested reader, since we exhibit in Theorem \ref{2083} a closed subspace $X\subset c_0$ for which $\K(X)$ carries two countably infinite chains of closed ideals, where the ideals are of a similar form as in the chains in \eqref{3o} and \eqref{4k}.
\end{remarks}

Let $2<p<\infty$. By \cite[Theorem 3.13]{TW22} there is a closed subspace $X\subset c_0$ such that the uniform closures $\overline{\QN_p(X)}$ and $\overline{\SK_p(X)}$ are incomparable closed ideals of $\mathcal L(X)$ such that
\begin{equation}\label{6}\A(X)\subsetneq \overline{\QN_p(X)}\subsetneq\K(X)\quad\text{ and }\quad\A(X)\subsetneq\overline{\SK_p(X)}\subsetneq\K(X).\end{equation}
This result is improved upon in Theorem \ref{2083} where we find a closed subspace $X\subset c_0$ that, in addition to \eqref{6}, carries two countably infinite chains of closed ideals of $\mathcal L(X)$, where the closed ideals of the first chain lie strictly between $\A(X)$ and $\overline{\QN_p(X)}$ and the closed ideals of the second chain lie strictly between $\A(X)$ and $\overline{\SK_p(X)}$. Here the closed ideals have the form  $\overline{\QN_{a_r}(X)^{b_r}}$ and $\overline{\SK_{a_r}(X)^{b_r}}$ for suitable increasing sequences $(a_r)\subset\mathbb N$ and $(b_r)\subset\mathbb N$ of natural numbers. Moreover, the closed ideals $\overline{\QN_{a_r}(X)^{b_r}}$ and $\overline{\SK_{a_s}(X)^{b_s}}$ are incomparable for every $r,s\in\mathbb N$.
It then follows from (BOI3) that the ideals $\overline{\QN_{a_r}(X)^{b_r}}^{\V_{\QN_p}}$ form a countably infinite chain of closed ideals of $\QN_p(X)$, and similarly, the ideals $\overline{\SK_{a_r}(X)^{b_r}}^{\V_{\SK_p}}$ form a countably infinite chain of closed ideals of $\SK_p(X)$. 
\smallskip

In order to prove the incomparability of the two chains of uniformly closed ideals described above we require the following refinement of Lemma \ref{minimal}.(i) for $\QN_p$ and Lemma \ref{minimal}.(ii) for $\SK_p$ in the case of $p>2$.

\begin{lem}\label{cor:minimal}
Let $2<p,q<\infty$ and let $X$, $Y$ and $Z$ be arbitrary Banach spaces. 
\begin{enumerate}
\item[(i)] Suppose that $S\in\QN_p(X,Y)$ and $T\in\overline{\SK_q(Y,Z)}$. Then $TS\in\overline{\F(X,Z)}^{\V_{\QN_p}}$.
\item[(ii)] Suppose that $S\in\overline{\QN_q(X,Y)}$ and $T\in\SK_p(Y,Z)$. Then $TS\in \overline{\F(X,Z)}^{\V_{\SK_p}}$. 
\end{enumerate}
\end{lem}
\begin{proof}
(i) According to the proof of \cite[Lemma 5]{PP69}, the operator $S\in\QN_p(X,Y)$ admits a factorisation $S=VU$ where $U\in\QN_p(X,M)$, $V\in\K(M,Y)$ and $M$ is a closed subspace of $\ell^p$. Since $T\in\overline{\SK_q(Y,Z)}$ by assumption, the operator ideal property yields that $TV\in \overline{\SK_q(M,Z)}$. Moreover, the closed subspace $M\subset\ell^p$ has type 2 and thus the dual space $M^*$ has cotype 2, see \cite[Proposition 11.10]{DJT95}. By Proposition \ref{QNpcotype2} the dual $M^*$ has the $\QN_q$-AP, and thus according to Fact \ref{16122} we have \[\SK_q(M,Z)=\overline{\F(M,Z)}^{\V_{\SK_q}}\subset\A(M,Z),\]
where the inclusion follows from $\V\le \V_{\SK_p}$. Consequently, $TV\in\overline{\SK_q(M,Z)}=\A(M,Z)$ and thus $TS=(TV)U\in\overline{\F(X,Z)}^{\V_{\QN_p}}$ by Lemma \ref{minimal}.(i).
\smallskip

(ii) According to \cite[Theorem 3.2]{SK02} (see also \cite[p. 2446]{CK10}), the operator $T\in\SK_p(Y,Z)$ admits a factorisation $T=BA$, where $A\in \mathcal L(Y,W)$, $B\in\SK_p(W,Z)$ and $W=\ell^{p'}/N$ is a quotient space of $\ell^{p'}$. Since $S\in\overline{\QN_q(X,Y)}$ by assumption, we have $AS\in\overline{\QN_q(X,W)}$ by the operator ideal property. Moreover, $W^*\cong N^\perp\subset\ell^p$ has type 2, and thus the reflexive space $W$ has cotype 2, see \cite[Proposition 11.10]{DJT95}. Consequently, $W$ has the $\QN_q$-AP by Proposition \ref{QNpcotype2}. It follows that 
\[\QN_q(X,W)=\overline{\F(X,W)}^{\V_{\QN_q}}\subset\A(X,W),\] 
and thus $AS\in\overline{\QN_q(X,W)}=\A(X,W)$. Finally, Lemma \ref{minimal}.(ii) yields that $TS=B(AS)\in \overline{\F(X,Y)}^{\V_{\SK_p}}$. 
\end{proof}

We are now in position to prove the main result of this section.
\begin{thm}\label{2083}
Suppose that $2<p<\infty$. Then there is a closed subspace $X\subset c_0$ together with two decreasing chains $\{\I_r\mid r\in\mathbb N\}$ and $\{\mathcal J_r\mid r\in\mathbb N\}$ of algebraic ideals of $\mathcal L(X)$ such that the following holds for the uniform norm closures:
\begin{equation}\label{eq:29}
\begin{tikzcd}[tips=false,column sep=0.01em,row sep=0.01em]
&\cdots\arrow[draw=none]{r}[sloped,auto=false]{\subsetneq}&\overline{\mathcal I_{r+1}}\arrow[draw=none]{r}[sloped,auto=false]{\subsetneq}&\overline{\mathcal I_r}\arrow[draw=none]{r}[sloped,auto=false]{\subsetneq}&\cdots\arrow[draw=none]{r}[sloped,auto=false]{\subsetneq} & \overline{\mathcal I_1}\arrow[draw=none]{r}[sloped,auto=false]{\subsetneq} &\overline{\QN_p(X)}\arrow[draw=none]{dr}[sloped,auto=false]{\subsetneq}&
\\          
\A(X)\arrow[draw=none]{dr}[sloped,auto=false]{\subsetneq}\arrow[draw=none]{ur}[sloped,auto=false]{\subsetneq}&&& & & &  &\K(X).\\   
&\cdots\arrow[draw=none]{r}[sloped,auto=false]{\subsetneq}&\overline{\mathcal J_{r+1}}\arrow[draw=none]{r}[sloped,auto=false]{\subsetneq}&\overline{\mathcal J_r}\arrow[draw=none]{r}[sloped,auto=false]{\subsetneq}&\cdots\arrow[draw=none]{r}[sloped,auto=false]{\subsetneq}  &\overline{\mathcal J_1} \arrow[draw=none]{r}[sloped,auto=false]{\subsetneq}& \overline{\SK_p(X)}\arrow[draw=none]{ur}[sloped,auto=false]{\subsetneq}&
\end{tikzcd}
\end{equation}
The two chains of closed ideals in \eqref{eq:29}  are incomparable in the sense that each closed ideal in the upper chain is incomparable with each closed ideal in the lower chain. Moreover, the following hold:
\begin{equation}\label{09014}
\overline{\F(X)}^{\V_{\QN_p}}\subsetneq\cdots \subsetneq\overline{\I_{r+1}}^{\V_{\QN_p}}\subsetneq \overline{\I_r}^{\V_{\QN_p}}\subsetneq \cdots\subsetneq \overline{\I_1}^{\V_{\QN_p}}
\subsetneq\QN_p(X),
\end{equation}
\begin{equation}\label{09013}
\overline{\F(X)}^{\V_{\SK_p}}\subsetneq\cdots \subsetneq\overline{\mathcal J_{r+1}}^{\V_{\SK_p}}\subsetneq \overline{\mathcal J_r}^{\V_{\SK_p}}\subsetneq \cdots\subsetneq \overline{\mathcal J_1}^{\V_{\SK_p}}
\subsetneq\SK_p(X).\end{equation}
In particular, by Proposition \ref{quotient} the compact-by-approximable algebra $\QA_X=\K(X)/\A(X)$ contains two incomparable countably infinite chains of closed ideals which are all nilpotent by Lemma \ref{16120}. Moreover, the quotient algebras $\QAQN$ and $\QASK$ both contain a countably infinite chain of closed ideals. 
\end{thm}

\begin{proof}
Let $\lambda:=\left\lceil \log_2\Big(\frac{p+1}{p-2}\Big)\right\rceil$. In order to simplify the notation, we denote $\theta(r)=2^{r+\lambda}$ for all $r\in\mathbb N_0=\mathbb N\cup\{0\}$. Clearly $\left\lceil \frac{\theta(r+1)+1}{2}\right\rceil=\theta(r)+1$ for all $r\in\mathbb N_0$. Thus by Proposition \ref{prop43} there is for all $r\in\mathbb N_0$ a closed subspace $M_r\subset c_0$ together with the following operators:
\begin{enumerate}[(i)]
\item $\widetilde T_r\in\QN_{\theta(r+1)+1}(M_r)$ such that $\widetilde T_r^{\theta(r)}\notin\A(M_r)$.
\item $\widetilde U_r\in\SK_{\theta(r+1)+1}(M_r)$ such that $\widetilde U_r^{\theta(r)}\notin\A(M_r)$.
\end{enumerate}
Moreover, by \cite[Theorem 2.9]{TW22} there is a closed subspace $M\subset c_0$ that carries a compact operator
\begin{enumerate}
\item[(iii)]  $\widetilde V\in\K(M)$ such that $\widetilde V^r\notin\A(M)$ for all $r\in\mathbb N$.
\end{enumerate}
Consider the closed subspace $X:=\big(M\oplus M_0\oplus M_1\oplus M_2\oplus\cdots\big)_{c_0}\subset c_0$. For each $r\in\mathbb N$ we define the following algebraic ideals of $\mathcal L(X)$:
\begin{align}\label{definitionofI_r}
\I_r&=\QN_{\theta(r+1)+1}(X)^{\theta(r)-1},\\
\label{definitionofJ_r}\mathcal J_r&=\SK_{\theta(r+1)+1}(X)^{\theta(r)-1}.\end{align}
It follows from Steps 1--7 below that the closed subspace $X\subset c_0$ together with the ideals $\I_r$ and $\mathcal J_r$ in \eqref{definitionofI_r} and \eqref{definitionofJ_r} has the desired properties. 
\smallskip

For Steps 1, 2 and 4 we define the auxiliary ideals $\mathcal I_0=\QN_{\theta(1)+1}(X)^{\theta(0)-1}$ and $\mathcal J_0=\SK_{\theta(1)+1}(X)^{\theta(0)-1}$ in a similar way as in \eqref{definitionofI_r} and \eqref{definitionofJ_r}. Note that $\theta(0)=2^\lambda\ge 2$ and thus $\I_0$ and $\J_0$ are well-defined ideals of $\mathcal L(X)$. 
\medskip

\textbf{Step 1.} \emph{Claim.} 
\begin{enumerate}
\item[(a)] $\I_0\subset\QN_p(X)$. 
\item[(b)] $\mathcal J_0\subset\SK_p(X)$.
\end{enumerate}

For part (a) suppose that $T\in \I_0$. By the multiplication rule \eqref{powermultiplication} for the quasi $p$-nuclear operators and linearity, we have $T\in\QN_q(X)$ for $q:=\frac{\theta(1)+1}{\theta(0)-1}$. Moreover, since $\lambda\ge \log_2(\frac{p+1}{p-2})$ one gets that
\[
q=\frac{\theta(1)+1}{\theta(0)-1}=\frac{2^{1+\lambda}+1}{2^{\lambda}-1}=2+\frac{3}{2^{\lambda}-1}\le 2+\frac{3}{\frac{p+1}{p-2}-1}=p.
\] 
Thus $T\in \QN_p(X)$ by the monotonicity  \eqref{monot} of the classes $\QN_r$.
\smallskip

Part (b) is verified in a similar way. In fact, suppose that $T\in\mathcal J_0$. Then $T\in\SK_q(X)$ by the multiplication rule for the Sinha-Karn $p$-compact operators \eqref{powermultiSKp}, where $q$ is the constant in part (a). Consequently, by the monotonicity \eqref{monotonicity} of the classes $\SK_r$, we have $T\in\SK_p(X)$. 
 
\medskip

\textbf{Step 2.} \emph{Claim.} For all $r\in\mathbb N_0=\mathbb N\cup\{0\}$ the following hold:
\begin{enumerate}
\item[(a)] $\I_{r+1}\subset\I_r$. 
\item[(b)] $\mathcal J_{r+1}\subset\mathcal J_r$.
\end{enumerate}
Fix $r\in\mathbb N_0$ for the argument. For part (a) suppose that $T\in\I_{r+1}$. By linearity we may assume that $T=S_1\cdots S_{\theta(r+1)-1}$, where $S_j\in\QN_{\theta(r+2)+1}(X)$ for all $j\in\{1,\ldots,\theta(r+1)-1\}$. Since $(\theta(r+2)+1)/2=\theta(r+1)+\frac{1}{2}$, it follows by the multiplication rule \eqref{powermultiplication} and the monotonicity \eqref{monot} that
\[S_jS_{j+1}\in\QN_{\theta(r+1)+\frac{1}{2}}(X)\subset\QN_{\theta(r+1)+1}(X)\]
for all $j\in\{1,\ldots,\theta(r+2)-2\}$. Thus
\[T=S_1\overbrace{(S_2S_3)\cdots(S_{\theta(r+1)-2}S_{\theta(r+1)-1})}^{\frac{\theta(r+1)-2}{2}=\theta(r)-1\text{ pairs}}\in\QN_{\theta(r+1)+1}(X)^{\theta(r)-1}=\I_r.\]

Part (b) is shown exactly in the same way by using the corresponding multiplication rule \eqref{powermultiSKp} and the monotonicity \eqref{monotonicity} for Sinha-Karn $p$-compact operators.
\medskip

Since $\QN_p\subset\K$, $\SK_p\subset\K$ and $\F(X)$ is the minimal ideal of $\mathcal L(X)$ (see Fact \ref{090909}), the following diagram holds by Steps 1 and 2:
\begin{equation}\label{eq:230}
\begin{tikzcd}[tips=false,column sep=0.01em,row sep=0.01em]
&\cdots\arrow[draw=none]{r}[sloped,auto=false]{\subset}&\mathcal I_{r+1}\arrow[draw=none]{r}[sloped,auto=false]{\subset}&\mathcal I_r\arrow[draw=none]{r}[sloped,auto=false]{\subset}&\cdots\arrow[draw=none]{r}[sloped,auto=false]{\subset} &\mathcal I_1\arrow[draw=none]{r}[sloped,auto=false]{\subset} &\QN_p(X)\arrow[draw=none]{dr}[sloped,auto=false]{\subset}&
\\          
\F(X)\arrow[draw=none]{dr}[sloped,auto=false]{\subset}\arrow[draw=none]{ur}[sloped,auto=false]{\subset}&&& & & &  &\K(X).\\   
&\cdots\arrow[draw=none]{r}[sloped,auto=false]{\subset}&\mathcal J_{r+1}\arrow[draw=none]{r}[sloped,auto=false]{\subset}&\mathcal J_r\arrow[draw=none]{r}[sloped,auto=false]{\subset}&\cdots\arrow[draw=none]{r}[sloped,auto=false]{\subset}  &\mathcal J_1 \arrow[draw=none]{r}[sloped,auto=false]{\subset}& \SK_p(X)\arrow[draw=none]{ur}[sloped,auto=false]{\subset}&
\end{tikzcd}
\end{equation}
Consequently, the inclusions in \eqref{eq:29},  \eqref{09014} and \eqref{09013} hold. We next show that all the inclusions in \eqref{eq:29}, \eqref{09014} and \eqref{09013} are strict.
\medskip

\textbf{Step 3.} \emph{Claim.} 

\begin{enumerate}
\item[(a)] $\overline{\QN_p(X)}\subsetneq \K(X)$. 
\item[(b)] $\overline{\SK_p(X)}\subsetneq \K(X)$.
\end{enumerate}
The claim follows from a similar argument as Step 1 of Theorem \ref{0301}. In fact, we first observe that the compact operator $\widetilde T\in\K(M)$ in (iii) together with the operator ideal property implies that $V:=J_M\widetilde TP_M\in\K(X)$ satisfies $V^r\notin\A(X)$ for every $r\in\mathbb N$. Thus $\QA_X=\K(X)/\A(X)$ is non-nilpotent. However, by Lemma \ref{16120} the quotients $\overline{\QN_p(X)}/\A(X)$ and $\overline{\SK_p(X)}/\A(X)$ are nilpotent and thus both claims in Step 3 hold.
\smallskip

For Steps 4 and 5 we define for every $r\in\mathbb N_0$ the operators $T_r:=J_r\widetilde T_r P_r$ and $U_r:=J_r\widetilde U_r P_r$, where $J_r:M_r\to X$ is the natural isometric embedding and $P_r:X\to M_r$ is the natural projection. By (i), (ii) and the operator ideal property, the following hold:
\begin{enumerate}
\item[(iv)] For all $r\in\mathbb N_0$ the operator $T_r\in\QN_{\theta(r+1)+1}(X)$ satisfies $T_r^{\theta(r)}\notin\A(X)$.
\item[(v)] For all $r\in\mathbb N_0$ the operator $U_r\in\SK_{\theta(r+1)+1}(X)$ satisfies $U_r^{\theta(r)}\notin\A(X)$.
\end{enumerate}
\textbf{Step 4.} \emph{Claim.} For all $r\in\mathbb N_0$ the following hold:
\begin{enumerate}
\item[(a)] $\overline{\I_{r+1}}\subsetneq \overline{\I_r}$. 
\item[(b)] $\overline{\mathcal J_{r+1}}\subsetneq \overline{\mathcal J_r}$.
\end{enumerate}

Fix $r\in\mathbb N_0$ for the argument. For part (a) note first that Step 2 implies $\overline{\I_{r+1}}\subset \overline{\I_r}$. In order to obtain the strict inclusion, we will show that
\begin{equation}\label{2687}
T_r^{\theta(r)-1}\in \I_r\setminus\overline{\I_{r+1}}.
\end{equation}
Firstly, it is clear that $T_r^{\theta(r)-1}\in\I_r$. Next, towards a contradiction, assume that $T_r^{\theta(r)-1}\in\overline{\I_{r+1}}$. Let $(S_n)\subset\I_{r+1}$ be a sequence such that $||S_n-T_r^{\theta(r)-1}||\to 0$ as $n\to\infty$. It follows that
\begin{equation}\label{U_m}
||S_nT_r-T_r^{\theta(r)}||=||(S_n-T_r^{\theta(r)-1})T_r||\le ||S_n-T_r^{\theta(r)-1}||\,||T_r||\to 0
\end{equation}
as $n\to\infty$.
By linearity and the multiplication rule \eqref{powermultiplication} we have $S_n\in \I_{r+1}\subset\QN_s(X)$  for every $n\in\mathbb N$, where $s:=\frac{\theta(r+2)+1}{\theta(r+1)-1}$. Moreover, since $T_r\in\QN_{\theta(r+1)+1}(X)$, it follows from \eqref{multiplication} that $S_nT_r\in\QN_t(X)$ for every $n\in\mathbb N$, where $\frac{1}{t}=\frac{1}{s}+\frac{1}{\theta(r+1)+1}$. Here $t< 2$, which is seen from the following computation (recall that $\theta(r+1)=2\cdot \theta(r)$ for all $r\in\mathbb N_0$):
\begin{align*}
\frac{1}{t}&=\frac{\theta(r+1)-1}{\theta(r+2)+1}+\frac{1}{\theta(r+1)+1}>\frac{\theta(r+1)-1}{\theta(r+2)+2}+\frac{1}{\theta(r+1)+1}\\
&=\frac{\theta(r+1)-1}{\theta(r+2)+2}+\frac{2}{\theta(r+2)+2}=\frac{\theta(r+1)+1}{2(\theta(r+1)+1)}=\frac{1}{2}.
\end{align*}
Consequently, $S_nT_r\in\A(X)$ for every  $n\in\mathbb N$ by \eqref{QNple2}. It follows from \eqref{U_m} that $T_r^{\theta(r)}\in\A(X)$, which is a contradiction in view of (iv). Thus \eqref{2687} holds, which yields the claim.
\smallskip

For part (b) observe first that Step 2 implies $\overline{\mathcal J_{r+1}}\subset\overline{\mathcal J_r}$. Now, the strict inclusion follows from the fact that
\[U_r^{\theta(r)-1}\in\mathcal J_r\setminus\overline{\mathcal J_{r+1}}.\]
This is shown in the same way as \eqref{2687} by using the corresponding multiplication rules \eqref{multip} and \eqref{powermultiSKp} for the classes $\SK_r$. We thus we omit the details.
\smallskip

By Steps 1, 3 and 4 and Fact \ref{090909} the diagram in \eqref{eq:29} holds. It is then straightforward to verify, by applying (BOI3), that also the strict inclusions in \eqref{09014} and \eqref{09013} hold. In fact, suppose for instance, that $\overline{\J_{m+1}}^{\V_{\SK_p}}\subsetneq\overline{\J_{m}}^{\V_{\SK_p}}$ does not hold for some $m\in\mathbb N$. This means in view of \eqref{eq:230} that $\overline{\J_{m+1}}^{\V_{\SK_p}}=\overline{\J_{m}}^{\V_{\SK_p}}$. But then $\overline{\J_{m+1}}=\overline{\J_m}$ since $||\cdot||\le ||\cdot||_{\SK_p}$ by (BOI3) which contradicts Step 4. The other strict inclusions in \eqref{09014} and \eqref{09013} are verified in the same way.
\smallskip

Finally, we show in Steps 5--7 that the two chains of closed ideals in \eqref{eq:29} are incomparable. 
\medskip

\textbf{Step 5.} \emph{Claim.} For all $r\in\mathbb N$ the following hold:
\begin{enumerate}
\item[(a)] $\overline{\I_r}\not\subset\overline{\SK_p(X)}$. 
\item[(b)] $\overline{\mathcal J_r}\not\subset\overline{\QN_p(X)}$. 
\end{enumerate}
Fix $r\in\mathbb N$ for the argument. We claim that  $T_r^{\theta(r)-1}\in\I_r\setminus\overline{\SK_p(X)}$, which then implies part (a). Clearly $T_r^{\theta(r)-1}\in\I_r$. Next, assume towards a contradiction that $T_r^{\theta(r)-1}\in\overline{\SK_p(X)}$. It follows by Lemma \ref{cor:minimal}.(i) that
\[T_r^{\theta(r)}=T_r^{\theta(r)-1}\cdot T_r\in\overline{\F(X)}^{\V_{\QN_{\theta(r+1)+1}}}\subset\A(X).\]
But this contradicts (iv). 

For part (b) it suffices to verify that $U_r^{\theta(r)-1}\in\mathcal J_r\setminus\overline{\QN_p(X)}$. Again, clearly $U_r^{\theta(r)-1}\in\mathcal J_r$. Next, assume on the contrary that $U_r^{\theta(r)-1}\in\overline{\QN_p(X)}$. By Lemma \ref{cor:minimal}.(ii) we have \[U_r^{\theta(r)}=U_r\cdot U_r^{\theta(r)-1}\in\overline{\F(X)}^{\V_{\SK_{\theta(r+1)+1}}}\subset\A(X),\] which contradicts (v). 
\medskip

\textbf{Step 6.} \emph{Claim.} 
\begin{enumerate}
\item[(a)] $\overline{\QN_p(X)}\not\subset\overline{\SK_p(X)}$.
\item[(b)] $\overline{\SK_p(X)}\not\subset\overline{\QN_p(X)}$.
\end{enumerate}

Both claims are direct consequences of Steps 1, 2 and 5. In fact, for (a), assume towards a contradiction that $\overline{\QN_p(X)}\subset\overline{\SK_p(X)}$. Then Steps 1 and 2 implies $\overline{\I_1}\subset\overline{\SK_p(X)}$, which contradicts part (a) of Step 5. A similar argument yields part (b).
\medskip

\textbf{Step 7.} \emph{Claim.} For all $r,s\in\mathbb N$ the following hold:
\begin{enumerate}
\item[(a)] $\overline{\I_r}\not\subset\overline{\mathcal J_s}$.
\item[(b)] $\overline{\mathcal J_s}\not\subset\overline{\I_r}$.
\end{enumerate}

Here also the claims follow immediately from Steps 1, 2 and 5. For part (a), assume  that $\overline{\I_r}\subset\overline{\mathcal J_s}$ for some $r,s\in\mathbb N$. Then Steps 1 and 2 imply $\overline{\I_r}\subset\overline{\SK_p(X)}$, which contradicts part (a) of Step 5. Again, a similar argument yields part (b). 
\end{proof}

We conclude this section with some additional observations.
\begin{remarks}
(i) In the case $2<p\le 4$ of Theorem \ref{2083} one obtains two chains of 2-nilpotent closed ideals of $\QA_X$, since here \[index(\overline{\QN_p(X)}/\A(X))=index(\overline{\SK_p(X)}/\A(X))=2\] 
by Lemma \ref{16120}.
\smallskip

(ii) Suppose that $1\le p\le 4$ and $p\neq 2$. Then any closed subspace of $\QA_X^{\QN_p}$ and of $\QASK$ is automatically a closed ideal since the quotient algebras $\QAQN$ and $\QASK$ are 2-nilpotent by Proposition \ref{prop41}.
\end{remarks}

\section{Concluding examples and problems}\label{section6}

In the final section of this paper we exhibit a construction of a Banach space $Z$ for which the quotient algebra $\QA_Z^\I$ carries an uncountable family of closed ideals. This is a variant of the construction in \cite[Theorem 4.5]{TW22} for the case $\I=\K$. Here the setting is rather general since we only require that the Banach operator ideal $\I=(\I,\VI)$ is contained in the ideal of the compact operators $\K=(\K,\V)$ and that the $\I$-AP fails a certain duality property.
We also draw attention to duality problems related to the Banach operator ideals $\QN_p=(\QN_p,\V_{\QN_p})$ and $\SK_p=(\SK_p,\V_{\SK_p})$, and the corresponding quotient algebras $\QAQN$ and $\QASK$ for $p\neq 2$.
\smallskip
 
We proceed with the example where $\QA_Z^\I$ carries an uncountable family of closed ideals for a specific class of Banach operator ideals $\I=(\I,\VI)$ such that $\I\subset\K$. In fact, we will require the existence of Banach spaces $X$ and $Y$ such that $X$ has the $\I$-AP and $\I(X,Y)\neq\overline{\F(X,Y)}^{\VI}$. We note that the latter condition typically concerns the failure of the $\I^{dual}$-AP of the dual space $X^*$. In fact, let $X$ be a Banach space and suppose that $\I=(\I,\VI)$ is a Banach operator ideal for which $\I=(\I^{dual})^{dual}$. Then one may verify (cf. \cite[Proposition 4.9]{Kim19}) that the following holds:
\begin{equation*}
X^*\text{ has the }\I^{dual}\text{-AP if and only if }\I(X,Y)=\overline{\F(X,Y)}^{\VI}\text{ for every dual space }Y.
\end{equation*}

\begin{ex}\label{infinitelymanyideals}
Fix $1< p<\infty $ and let $\I=(\I,\VI)$ be a Banach operator ideal such that $\I\subset\K$. Suppose that the Banach spaces $X$ and $Y$ satisfy the following two conditions: 
\begin{enumerate}[(i)] 
\item $X$ has the $\I$-AP.

\item $\I(X,Y)\neq \overline{\F(X,Y)}^{\VI}$. 
\end{enumerate}
Set $X_0=Y$ and $X_n=X$ for all $n\in\mathbb N$ and consider the direct $\ell^p$-sum $Z:=\big(\bigoplus_{n=0}^\infty X_n\big)_{\ell^p}$. For each non-empty subset $A\subset\mathbb N$ we define
\[\I_A:=\left\{U \in \I(Z)\mid P_0UJ_k \in \overline{\F(X_k,X_0)}^{\VI} \text{ for all } k \in A\cup\{0\}\right\},\]
where $J_k:X_k\to Z$ is the natural isometric embedding and $P_k:Z\to X_k$ is the natural projection for each $k\in \mathbb N_0=\mathbb N\cup\{0\}$. 
\smallskip

\emph{Claim:} The family $\{\I_A\mid \emptyset\neq A\subset\mathbb N\}$ is an uncountable family of closed ideals of $\I(Z)$ that has the reverse partial order structure of the power set $(\mathcal P(\mathbb N),\subset)$ of the natural numbers $\mathbb N$.
\smallskip

Consequently, by Proposition \ref{quotient} the family $\{q(\I_A)\mid \emptyset\neq A\subset\mathbb N\}$ of closed ideals in the quotient algebra $\QA_Z^\I$ is uncountable and has the reverse partial order structure of the power set $(\mathcal P(\mathbb N),\subset)$. Here $q:\I(Z)\to \QA_Z^\I$ denotes the quotient map.
\smallskip

We establish the claim in four steps.
\smallskip

\textbf{Step 1.} \emph{Claim.} $\I_A$ is a closed subspace of $\I(Z)$ for every non-empty subset $A\subset \mathbb N$. 
\smallskip

Towards this,  let $\emptyset\neq A\subset \mathbb N$. Clearly $\I_A$ is linear subspace of $\I(Z)$. To verify that $\I_A=\overline{\I_A}^{\VI}$, suppose that $T\in\overline{\I_A}^{\VI}$. Let $\varepsilon>0$ be arbitrary and suppose that $k\in A\cup\{0\}$. By assumption, there is an operator $S\in\I_A$ such that $||T-S||_\I<\varepsilon$. Then $P_0SJ_k\in\overline{\F(X_k,X_0)}^{\VI}$ and since
\[
||P_0TJ_k-P_0SJ_k||_\I\le ||T-S||_\I<\varepsilon
\]
by (BOI2), it follows that $P_0TJ_k\in \overline{\F(X_k,X_0)}^{\VI}$. Thus $T\in\I_A$.
\medskip

\textbf{Step 2.} \emph{Claim.} $VU\in \I_A$ for every $U\in \I_A$ and $V\in \I(Z)$.
\medskip

Let $k\in A\cup\{0\}$ be fixed and suppose that $U\in \I_A$ and $V\in\I(Z)$. We first observe that for each $N\in\mathbb N$ we have
\begin{equation}\label{rre}
\sum_{n=0}^N P_0VJ_nP_nUJ_k\in\overline{\F(X_k,X_0)}^{\VI}.
\end{equation}
In fact, $P_0UJ_k\in\overline{\F(X_k,X_0)}^{\VI}$ since $U\in \I_A$. Moreover, $P_n UJ_k\in \overline{\F(X_k,X_n)}^{\VI}$ for all $n\in\{1,\ldots,N\}$, since in this case $X_n=X$ and $X$ has the $\I$-AP by assumption (i). Thus \eqref{rre} holds by the operator ideal property.

Next, note that $\sum_{n=0}^N J_nP_n\to I_Z$ pointwise as $N\to\infty$ and thus 
\begin{equation}\label{622022}
\lim_{N\to\infty}||U-\sum_{n=0}^N J_nP_n U||=0\end{equation} 
since $U\in\I(Z)\subset\K(Z)$.
Consequently, we have
\begin{align*}
||P_0VUJ_k-\sum_{n=0}^N P_0VJ_nP_nUJ_k||_\I&=||P_0V(U-\sum_{n=0}^N J_nP_nU)J_k||_\I\\
&\le ||U-\sum_{n=0}^N J_nP_n U||\cdot ||V||_\I\to 0
\end{align*}
as $N\to\infty$ by (BOI2). Thus $P_0VUJ_k\in\overline{\F(X_k,X_0)}^{\VI}$ by \eqref{rre}, which then yields that $VU\in I_A$.
\medskip

\textbf{Step 3.} \emph{Claim.} $UV\in \I_\mathbb N$ for every $U\in \I_A$ and $V\in \I(Z)$. In particular, $UV\in \I_A$.
\medskip

The claim follows from a similar argument as in Step 2. Here we have for a fixed $k\in\mathbb N_0$ and all $N\in\mathbb N$ that 
\begin{equation}\label{63220_a}
\sum_{n=0}^N P_0UJ_nP_nVJ_k\in \overline{\F(X_k,X_0)}^{\VI}.
\end{equation}
In fact, since $U\in\I_A$, we have $P_0UJ_0\in\overline{\F(X_0)}^{\VI}$. Moreover, since $V\in\I(Z)$ we have $P_nVJ_k\in\overline{\F(X_k,X_n)}^{\VI}$ for all $n\in\{1,\ldots,N\}$ by the operator ideal property and assumption (i). 

Now, as in \eqref{622022} we have that $||V-\sum_{n=0}^N J_nP_nV||\to 0$ as $N\to\infty$, since $V$ is a compact operator. Thus
\begin{align*}
||P_0UVJ_k-\sum_{n=0}^N P_0UJ_nP_nVJ_k||_\I&=||P_0U(V-\sum_{n=0}^N J_nP_nV)J_k||_\I\\
&\le ||V-\sum_{n=0}^N J_nP_n V||\cdot ||U||_\I\to 0
\end{align*}
as $N\to\infty$. It follows from \eqref{63220_a} that $P_0UVJ_k\in\overline{\F(X_k,X_0)}^{\VI}$, which yields the claim.
\medskip

\textbf{Step 4}. \emph{Claim.} For all non-empty subsets $A,B\subset\mathbb N$ the following holds:
\[A\subset B\text{ if and only if }\I_B\subset \I_A.\]
Let $A$ and $B$ be non-empty subsets of $\mathbb N$. Clearly the forward implication holds. For the converse implication, assume that $\I_B\subset \I_A$. Towards a contradiction, assume $A\not\subset B$ so that there is an element $r\in A\setminus B$. Then, by assumption (ii), there is an operator \[T\in\I(X_r,X_0)\setminus\overline{\F(X_r,X_0)}^{\VI}.\] Consider the operator $U:=J_0TP_r\in \I(Z)$. For a fixed $k\in B\cup\{0\}$ note that \[P_0UJ_k=P_0J_0TP_rJ_k=TP_rJ_k=0\] since $r\neq k$. Consequently, $U\in \I_B$. Moreover, $P_0UJ_r=T\notin\overline{\F(X_r,X_0)}^{\VI}$, which implies that $U\notin \I_A$. Thus $U\in\I_B\setminus\I_A$, which contradicts the initial assumption $\I_B\subset\I_A$. Consequently, $A\subset B$.  
\medskip

By Steps 1--3, the subset $\I_A$ is a closed ideal of $\I(Z)$ for every $\emptyset\neq A\subset\mathbb N$. Moreover, by Step 4 the family 
$\{\I_A\mid \emptyset\neq A\subset\mathbb N\}$ is uncountable and has the reverse partial order structure of $(\mathcal P(\mathbb N),\subset)$.
\end{ex}
\begin{remarks}\label{remarks15}
(i) The case $\I=\K$ of Example \ref{infinitelymanyideals} was established in \cite[Theorem 4.5]{TW22}. In that case one obtains Banach spaces $X$ and $Y$ that satisfy the conditions (i) and (ii) by considering a Banach space $X$ that has the AP for which the dual $X^*$ fails the AP (recall that such spaces exist, see \cite[Theorem 1.e.7(b)]{LT77}). According to \cite[Theorem 1.e.5]{LT77}, there is a  Banach space $Y$ such that $\K(X,Y)\neq\A(X,Y)$. Note here that the Banach space $X$ is necessarily non-reflexive by the duality property of the AP (see \cite[Theorem 1.e.7(a)]{LT77}), and consequently, also the direct $\ell^p$-sum $Z=(Y\oplus X\oplus X\oplus\cdots)_{\ell^p}$ in Example \ref{infinitelymanyideals} is non-reflexive.

Moreover, in the case $\I=\K$ the ideals $\I_A$ in Example \ref{infinitelymanyideals} have additional features as was observed in \cite{TW22}. For instance, the ideals $\I_A$ are left ideals of the bounded operators $\mathcal L(Z)$, since the claim of Step 2 holds for an arbitrary $V\in\mathcal L(Z)$ if $\V_{\I}=\V$. In addition, it follows from \cite[Lemma 4.6]{TW22} that $\I_{\mathbb N}=\A(X)$ and thus the closed ideals $q(\I_A)$ in $\QA_X$ are 2-nilpotent by Step 3. Here $q$ denotes the quotient map from $\K(X)$ onto $\QA_X$.
\smallskip

(ii) In contrast to the closed ideals found in Section \ref{section5}, the ideals $\I_A$ in Example \ref{infinitelymanyideals} are not algebraic ideals of the bounded operators $\mathcal L(Z)$ for any $\emptyset \neq A\subsetneq\mathbb N$, which can be seen with a similar argument as in \cite[Remark 4.8]{TW22}. In fact, for such a subset $A$, pick $r\in\mathbb N\setminus A$ and $s\in A$. By assumption (ii) in Example \ref{infinitelymanyideals}, there is an operator \[U\in\I(X_r,X_0)\setminus\overline{\F(X_r,X_0)}^{\VI}.\] 
Let $T:=J_0UP_r$ so that $T\in \I_A$ since $r\notin A\cup\{0\}$. Next, let $I_{sr}$ (respectively, $I_{rs}$) denote the identity operator on $X$ considered as an operator $X_s\to X_r$ (respectively, $X_r\to X_s$).  We claim that $TV\notin \I_A$ for the operator $V:=J_rI_{sr}P_s\in\mathcal L(Z)$, which shows that $\I_A$ is not an ideal of $\mathcal L(Z)$. 

For this, note that it suffices (by the definition of the ideal $\I_A$) to show that $P_0TVJ_s\notin\overline{\F(X_s,X_0)}^{\VI}$. But this holds by the operator ideal property since 
\[P_0TVJ_sI_{rs}=P_0(J_0UP_r)(J_rI_{sr}P_s)J_sI_{rs}=UI_{sr}I_{rs}=U\notin \overline{\F(X_r,X_0)}^{\VI}.\] 
Consequently, $TV\notin \I_A$ and thus $\I_A$ is not an ideal of $\mathcal L(Z)$.
\end{remarks}

In the next example we observe for any $2<p<\infty$ that there are Banach spaces $X$ and $Y$ such that the conditions (i) and (ii) in Example \ref{infinitelymanyideals} are satisfied for $\I=\QN_p$. Here both spaces $X$ and $Y$ can be chosen to be reflexive (in contrast to the case $\I=\K$, see Remarks \ref{remarks15}.(i)).
\begin{ex}
Suppose that $2<p<\infty$ and let $M\subset\ell^p$ be a closed subspace that fails the $\SK_p$-AP, see Example \ref{Ex:34}.(i). Consider the dual space $X:=M^*$. Since $M$ has type 2, the dual space $X$ has cotype 2, see e.g. \cite[Proposition 11.10]{DJT95}, and thus $X$ has the $\QN_p$-AP by Proposition \ref{QNpcotype2}. Moreover, since $M=X^*$ fails the $\SK_p$-AP, there is according to Fact \ref{1612} a Banach space $Y$ and an operator $T\in\QN_p(X,Y)\setminus\overline{\F(X,Y)}^{\V_{\QN_p}}$. By applying \cite[Lemma 5]{PP69} and the operator ideal property, we may assume that  $Y$ is a closed subspace of $\ell^p$. Consequently, since $X$ has the $\QN_p$-AP and $\QN_p(X,Y)\neq\overline{\F(X,Y)}^{\V_{\QN_p}}$, the reflexive Banach spaces $X$ and $Y$ satisfy conditions (i) and (ii) of Example \ref{infinitelymanyideals}. Note that the direct $\ell^p$-sum $Z=(Y\oplus X\oplus X\oplus X\cdots)_{\ell^p}$ in Example \ref{infinitelymanyideals} is then also reflexive.
\end{ex}
We conclude this paper with a few observations and problems related to the duality between the quotient algebras $\QA_X^{\QN_p}$ and $\QA_{X^*}^{\SK_p}$ as well as between $\QA_X^{\SK_p}$ and $\QA_{X^*}^{\QN_p}$ for $p\neq 2$. For this, we consider the following natural mappings for any Banach space $X$ and $p\in[1,\infty)\setminus\{2\}$:
\begin{align}
\label{PHI}&\Phi_X:\QA_X^{\QN_p}\to \QA_{X^*}^{\SK_p},\quad \Phi_X\big(T+\overline{\F(X)}^{\V_{\QN_p}}\big)= T^*+\overline{\F(X^*)}^{\V_{\SK_p}},\\
\label{PSI}&\Psi_X:\QA_X^{\SK_p}\to \QA_{X^*}^{\QN_p},\quad \Psi_X\big(T+\overline{\F(X)}^{\V_{\SK_p}}\big)= T^*+\overline{\F(X^*)}^{\V_{\QN_p}}.
\end{align}
Due to the dualities \eqref{911} and \eqref{9112}, these mappings are well-defined bounded linear antihomomorphisms, by which we mean that the mappings $\Phi_X$ and $\Psi_X$ reverse products. In fact, \[\Phi_X(ST+\overline{\F(X)}^{\V_{\QN_p}})=T^*S^*+\overline{\F(X^*)}^{\V_{\SK_p}}\] for all $S,T\in\QN_p(X)$ since $(ST)^*=T^*S^*$, and the analogue holds for the mapping $\Psi_X$.

By applying the dualities \eqref{911} and \eqref{9112}, it is easy to verify that if $X$ is a reflexive Banach space, then both mappings $\Phi_X$ and $\Psi_X$ are in addition isometric bijections. However, this is not always the case for non-reflexive Banach spaces and the following observation yields a non-reflexive Banach space $X$ for which the mappings $\Phi_X$ and $\Psi_X$ are not surjective.

\begin{prop}\label{prop64}
Suppose that $1\le p<\infty$ and $p\neq 2$. There exists a non-reflexive Banach space $W$ such that $\QA_W^{\QN_p}=\{0\}$ and $\QA_W^{\SK_p}=\{0\}$, but for which  $\QA_{W^*}^{\SK_p}\neq\{0\}$ and $\QA_{W^*}^{\QN_p}\neq\{0\}$.
\end{prop}
\begin{proof}
 Let $X$ be a separable reflexive Banach space such that 
\begin{equation}\label{222}
\QA_{X}^{\QN_p}\neq \{0\}\quad\text{ and }\quad\QASK\neq\{0\}.
\end{equation}
The construction of Reinov in \cite[Lemma 1.1]{Reinov82} together with the dualities displayed in \eqref{911} and \eqref{9112} ensure that such a Banach space exists, and we will verify this separately in Lemma \ref{finallemma} below. According to the James-Lindenstrauss construction (see \cite[Theorem 1.d.3]{LT77}), there is a Banach space $Z$ such that $Z^{**}$ has a Schauder basis and $Z^{**}/Z\approx X^*$. 
It follows that $Z^{***}\approx Z^*\oplus X$ since $X$ is reflexive.

Consider the bidual $W:=Z^{**}$. Since $W$ has the AP, it has both the $\QN_p$-AP and the $\SK_p$-AP (see Remarks \ref{9113}.(i)). Consequently, $\QA_{W}^{\QN_p}=\{0\}$ and $\QA_W^{\SK_p}=\{0\}$. On the other hand, there are operators $T\in\QN_p(X)\setminus\overline{\F(X)}^{\V_{\QN_p}}$ and $U\in\SK_p(X)\setminus\overline{\F(X)}^{\V_{\SK_p}}$ by \eqref{222}. Since $W^*=Z^{***}\approx Z^*\oplus X$, the following hold by the operator ideal property:
\begin{align*}
&\theta^{-1}J_XTP_X\theta\in\QN_p(W^*)\setminus\overline{\F(W^*)}^{\V_{\QN_p}},\\
&\theta^{-1}J_XUP_X\theta\in\SK_p(W^*)\setminus\overline{\F(W^*)}^{\V_{\SK_p}},
\end{align*}
where $J_X:X\to Z^*\oplus X$ and $P_X:Z^*\oplus X\to X$ are the natural maps, and $\theta:W^*\to Z^*\oplus X$ is a linear isomorphism. Consequently, $\QA_{W^*}^{\QN_p}\neq\{0\}$ and $\QA_{W^*}^{\SK_p}\neq\{0\}$. 
\end{proof}
In Proposition \ref{prop64} above we used the fact that there is a separable reflexive Banach space $X$ that satisfies \eqref{222}. We verify this in the following lemma.  
\begin{lem}\label{finallemma}
Suppose that $1\le p<\infty$ and $p\neq 2$. Then there is a separable reflexive Banach space $X$ such that $\QA_X^{\QN_p}\neq\{0\}$ and $\QA_X^{\SK_p}\neq\{0\}$.
\end{lem}
\begin{proof}
According to \cite[Lemma 1.1 and Corollary 1.1]{Reinov82}, there are separable reflexive Banach spaces $E$ and $F$ such that $\QN_p(E,F)\neq \overline{\F(E,F)}^{\V_{\varPi_p}}$. Since $\overline{\F}^{\V_{\QN_p}}=\overline{\F}^{\V_{\varPi_p}}$ (see \eqref{2012}) there is an operator $T\in\QN_p(E,F)\setminus\overline{\F(E,F)}^{\V_{\QN_p}}$. By reflexivity and the dualities \eqref{911} and \eqref{9112}, we have $T^*\in\SK_p(F^*,E^*)\setminus\overline{\F(F^*,E^*)}^{\V_{\SK_p}}$.

Consider the separable and reflexive direct $\ell^2$-sum $X:=(E\oplus F\oplus E^*\oplus F^*)_{\ell^2}$. Then $J_FTP_E\in\QN_p(X)\setminus\overline{\F(X)}^{\V_{\QN_p}}$ and $J_{E^*}T^*P_{F^*}\in\SK_p(X)\setminus\overline{\F(X)}^{\V_{\SK_p}}$ by the operator ideal property. This completes the proof. 
\end{proof}
In connection to Proposition \ref{prop64}, it appears unknown whether or not there is a non-reflexive Banach space $W$ (respectively, $Z$) such that $\QA_{W}^{\QN_p}\neq\{0\}$ and $\QA_{W^*}^{\SK_p}=\{0\}$ (respectively, $\QA_Z^{\SK_p}\neq\{0\}$ and $\QA_{Z^*}^{\QN_p}=\{0\}$) for $p\neq 2$. This suggests the following question related to the natural mappings displayed in \eqref{PHI} and \eqref{PSI}.
\begin{qu}\label{QUESTION}
Suppose that $1\le p<\infty$ and $p\neq 2$. Are the mappings $\Phi_X$ and $\Psi_X$ in \eqref{PHI} and \eqref{PSI} injective for all Banach spaces $X$ ?
\end{qu}
We note here that if, for instance, $\Psi_X$ is indeed always injective, then $\QA_{X^*}^{\QN_p}=\{0\}$ would imply $\QASK=\{0\}$. In this event Theorem \ref{trivialquotient}.(ii) would be a direct consequence of Theorem \ref{trivialquotient}.(i).

In connection to Question \ref{QUESTION}, we recall that a Banach operator ideal $\I=(\I,\VI)$ is called \emph{regular} if 
\[\I(X,Y)=\{T\in\mathcal L(X,Y)\mid j_Y T\in\I(X,Y^{**})\}\]
for all Banach spaces $X$ and $Y$, and $||T||_{\I}=||j_YT||_{\I}$ for all $T\in\I(X,Y)$, where $j_Y:Y\to Y^{**}$ is the canonical isometric embedding.
Due to the dualities \eqref{911} and \eqref{9112}, the following question on regularity is relevant for Question \ref{QUESTION}. 
\begin{qu}
Suppose that $1\le p<\infty$ and $p\neq 2$. Are the approximative kernels $\overline{\F}^{\V_{\QN_p}}$ and $\overline{\F}^{\V_{\SK_p}}$ regular Banach operator ideals ?
\end{qu}
It is known that some of the classical approximative Banach operator ideals are non-regular. For instance, the Banach operator ideal $\mathcal N_1=(\mathcal N_1,\V_{\mathcal N_1})$ of the 1-nuclear operators is not regular, which follows from a classical result of Figiel and Johnson \cite[Proposition 3]{FJ73}, see also \cite[8.3.8]{Pietsch80}. Moreover, results of Reinov in \cite[Section 3]{Reinov82} (see also \cite[Corollary 1.8]{ReinovReinov}) show that the Banach operator ideal $\N_p=(\N_p,\V_{\N_p})$ of the $p$-nuclear operators is not regular for any $1\le p<\infty$ and $p\neq 2$. We refer to to the paper \cite{ReinovReinov} for more results and discussions on (non-)regularity of classical approximative Banach operator ideals.

\section*{Acknowledgements} This work is part of the Ph.D.-thesis of the author and he would like to thank his supervisor Hans-Olav Tylli for suggesting the topic of this paper and for many valuable discussions.

This work was supported by the Magnus Ehrnrooth Foundation and the Swedish Cultural Foundation in Finland. 

\bibliographystyle{amsplain}

\bibliography{bibliographyWirzenius}
\end{document}